\documentclass[12pt]{amsart}

\usepackage{psfrag}
\usepackage{color}
\usepackage{tikz}
\usetikzlibrary{matrix,arrows}
\usepackage{graphicx,graphics}
\usepackage{fullpage,amssymb,amsfonts,amsmath,amstext,amsthm,amscd,verbatim}
\usepackage[T1]{fontenc}

\begin{document}

\newtheorem{theorem}{Theorem}[section]
\newtheorem{result}[theorem]{Result}
\newtheorem{fact}[theorem]{Fact}
\newtheorem{conjecture}[theorem]{Conjecture}
\newtheorem{lemma}[theorem]{Lemma}
\newtheorem{proposition}[theorem]{Proposition}
\newtheorem{remark}[theorem]{Remark}
\newtheorem{corollary}[theorem]{Corollary}
\newtheorem{facts}[theorem]{Facts}
\newtheorem{props}[theorem]{Properties}
\newtheorem*{thm3}{Theorem \ref{s2t3}}
\newtheorem*{thm4}{Theorem \ref{s2t4}}
\newtheorem*{propB}{Proposition \ref{Blaschke}}
\newtheorem*{lM1}{Lemma \ref{sMl1}}
\newtheorem*{lM2}{Lemma \ref{sMl2}}
\newtheorem*{tM3}{Theorem \ref{sMl3}}
\newtheorem{ex}[theorem]{Example}
\theoremstyle{definition}
\newtheorem{definition}[theorem]{Definition}

\newcommand{\notes} {\noindent \textbf{Notes.  }}
\newcommand{\note} {\noindent \textbf{Note.  }}
\newcommand{\defn} {\noindent \textbf{Definition.  }}
\newcommand{\defns} {\noindent \textbf{Definitions.  }}
\newcommand{\x}{{\bf x}}
\newcommand{\z}{{\bf z}}
\newcommand{\B}{{\bf b}}
\newcommand{\V}{{\bf v}}
\newcommand{\T}{\mathbb{T}}
\newcommand{\Z}{\mathbb{Z}}
\newcommand{\Hp}{\mathbb{H}}
\newcommand{\D}{\mathbb{D}}
\newcommand{\R}{\mathbb{R}}
\newcommand{\N}{\mathbb{N}}
\renewcommand{\B}{\mathbb{B}}
\newcommand{\C}{\mathbb{C}}
\newcommand{\ft}{\widetilde{f}}
\newcommand{\dt}{{\mathrm{det }\;}}
 \newcommand{\adj}{{\mathrm{adj}\;}}
 \newcommand{\0}{{\bf O}}
 \newcommand{\av}{\arrowvert}
 \newcommand{\zbar}{\overline{z}}
 \newcommand{\xbar}{\overline{X}}
 \newcommand{\htt}{\widetilde{h}}
\newcommand{\ty}{\mathcal{T}}
\renewcommand\Re{\operatorname{Re}}
\renewcommand\Im{\operatorname{Im}}
\newcommand{\tr}{\operatorname{Tr}}

\newcommand{\ds}{\displaystyle}
\numberwithin{equation}{section}
\numberwithin{figure}{section}

\renewcommand{\theenumi}{(\roman{enumi})}
\renewcommand{\labelenumi}{\theenumi}

\title{Dynamics of mappings with constant dilatation.}

\author{Alastair Fletcher \& Robert Fryer}

\begin{abstract}
Let $h:\C \to \C$ be an $\R$-linear map. In this article, we explore the dynamics of the quasiregular mapping $H(z)=h(z)^2$. Via the B\"{o}ttcher type coordinate constructed in \cite{FF}, we are able to obtain results for any degree two mapping of the plane with constant complex dilatation. We show that any such mapping has either one, two or three fixed external rays, that all cases can occur, and exhibit how the dynamics changes in each case.
We use results from complex dynamics to prove that these mappings are nowhere uniformly quasiregular in a neighbourhood of infinity.
We also show that in most cases, two such mappings are not quasiconformally conjugate on a neighbourhood of infinity.

\end{abstract}

\maketitle

\section{Introduction}
The recent interest in complex dynamics was initiated by Douady and Hubbard in the 1980s, following on from the work of Fatou, Julia, Montel and others towards the beginning of the twentieth century. Douady and Hubbard initially focussed on the iteration of quadratic polynomials and, with the advent of computer generated images, showed how the iteration of a simply stated function could have very complicated behaviour. For an overview of this theory, see for example \cite{CG,Milnor}.

A new avenue of research into the iteration of quasiregular mappings has recently opened up. Informally, quasiregular mappings send infinitesimal circles to infinitesimal ellipses and the greater the eccentricity of the ellipses, the greater the distortion of the mapping. Recall that a holomorphic function sends infinitesimal circles to infinitesimal circles. Inspired by the value distribution properties that quasiregular mappings have in analogy with holomorphic functions, there has been investigation into the iteration of quasiregular mappings in $\R^n$.

The first quasiregular mappings to be iterated were uniformly quasiregular mappings \cite{IM}, with a uniform bound on the distortion of the iterates. These are special mappings, and in particular, every uniformly quasiregular mapping of the plane is a quasiconformal conjugate of a holomorphic function \cite{Hinkkanen}.

In this article, we will study the iteration of quasiregular mappings in the plane with degree $2$ and constant complex dilatation. The iteration of such mappings was studied in \cite{FG}, where it was shown that every such mapping is linearly conjugate to a mapping of the form $h(z)^2+c$, where $h$ is an affine stretch and $c \in \C$. In our previous work \cite{FF}, we constructed a B\"{o}ttcher coordinate for these mappings in a neighbourhood of infinity. This coordinate is a quasiconformal map which conjugates $h(z)^2+c$ to $h(z)^2$ near infinity.

Therefore we will study in detail the dynamics of maps of the form $h(z)^2$. It was already shown in \cite{FF} that such mappings are not uniformly quasiregular, and hence are not quasiconformal conjugates of holomorphic functions, by looking at the behaviour on a fixed ray. We will further explore the properties of the fixed rays of these mappings, and what they tell us about the dynamical properties of these mappings. Along the way, we will need to make use of results from the theory of the iteration of M\"{o}bius maps and Blaschke products in a pleasing example of using complex dynamics to prove results in quasiregular dynamics.

\section{Statement of results}

Let $K > 1$ and $\theta \in (-\pi/2, \pi/2]$. Write $h_{K,\theta}$ for the affine stretch by factor $K$ in direction $e^{i\theta}$, that is,
\[ h_{K,\theta} := \left ( \frac{K+1}{2} \right ) z + e^{2i\theta} \left ( \frac{K-1}{2} \right ) \overline{z}.\]
If $K=1$, then this mapping is the identity and does not depend on $\theta$.
The purpose of this article is to continue the study of the dynamics of the quasiregular mappings $h(z)^2+c$ initiated in \cite{FG},
where $h=h_{K,\theta}$ for $K > 1$ and $\theta \in (\pi/2, \pi/2]$, and $c \in \C$.
If the mapping $h$ is fixed, we will write $H(z)=h(z)^2$.

The justification for studying these mappings in the class of degree two quasiregular mappings of the plane with constant complex dilatation is given by the following proposition.

\begin{proposition}
\label{canonicalform}
Let $f:\C \to \C$ be quasiregular of degree two and let $f$ have constant complex dilatation that is not identically $0$. Then $f$ is linearly conjugate to a unique mapping of the form
$h_{K,\theta}(z)^2+c$ for some $K> 1$, $\theta \in (-\pi/2, \pi/2]$ and $c \in \C$.
\end{proposition}

We prove this proposition in the next section. A \emph{ray} is a semi-infinite line $R_{\phi} = \{ te^{i\phi} : t \geq 0\}$. It is not hard to see that $h_{K,\theta}$ maps rays to rays, and so $H$ also maps rays to rays. This means that $H$ induces an increasing mapping $\widetilde{H}:\R \to \R$ that is $2\pi$-periodic, and is given by $\widetilde{H}(\varphi) = \arg[H(re^{i\varphi})]$, for any $r>0$. We say that a ray $R_{\phi}$ which is fixed by $H$ is \emph{locally repelling}, \emph{locally expanding} or \emph{neutral} if the induced mapping satisfies $\widetilde{H}'(\phi) <1$, $\widetilde{H}'(\phi) >1$ or $\widetilde{H}'(\phi)=1$ respectively.

\begin{theorem}\label{s2t1}
Let $\theta \in (-\pi/2, \pi/2]$, $K>1$ and let $H(z)=h_{K,\theta}(z)^2$. Then there exists $K_\theta>1$ such that:
\begin{itemize}
  \item for $K<K_{\theta}$, there is one fixed ray that is locally repelling;
  \item for $K=K_{\theta}$, there are two fixed rays, one of which is locally repelling and one that is neutral. Further, the neutral fixed ray is repelling on one side and attracting on the other;
  \item for $K>K_{\theta}$, there are three fixed rays, one of which is locally attracting and two that are locally repelling.
\end{itemize}
\end{theorem}

We next investigate the pre-images of these fixed rays. If $H$ has two or three fixed rays, then denote by $\Lambda$ the basin of attraction of the fixed ray that is not locally repelling.

\begin{theorem}\label{s2t2}
If $H$ has one fixed ray $R_\phi$ then $\{H^{-k}(R_\phi)\}_{k=0}^{\infty}$ is dense in $\C$. If $H$ has two or three fixed rays,
then $\Lambda$ is dense in $\C$.
\end{theorem}

Via the B\"{o}ttcher coordinate constructed in \cite{FF}, these theorems have analogues for mappings of the form $h(z)^2+c$ for $c \in \C$. We first make the following definition in analogy with complex dynamics.

\begin{definition}
\label{externalrays}
Let $f(z) = h(z)^2+c$. Then the \emph{external ray} $E_{\varphi}$ of $f$ with angle $\varphi \in [0,2\pi)$ is given by the image of the ray $R_{\varphi}$ under the quasiconformal B\"{o}ttcher coordinate $\psi = \psi(K,\theta, c)$ which conjugates $f$ to $H$. The external ray $E_{\varphi}$ is only defined in the range of $\psi$, that is, a neighbourhood of infinity.
\end{definition}

We remark that each $E_{\varphi}$ is an asymptotically conformal arc of a quasi-circle, since the B\"{o}ttcher coordinate is asymptotically conformal as $|z| \to \infty$. The collection $\{E_{\varphi} : \varphi \in [0,2\pi)\}$ foliates a neighbourhood of infinity. We define an external ray $E_{\varphi}$ which is fixed by $f$ to be attracting, repelling or neutral if the corresponding fixed ray $R_{\varphi}$ of $H$ is attracting, repelling or neutral respectively. The following corollary is an immediate application of Theorem \ref{s2t1}.

\begin{corollary}
Let $\theta \in (-\pi/2, \pi/2]$, $K>1$, $c \in \C$ and let $f(z)=h_{K,\theta}(z)^2+c$. Then, with $K_{\theta}$ as in Theorem \ref{s2t1},
\begin{itemize}
  \item for $K<K_{\theta}$, there is one fixed external ray of $f$ that is locally repelling;
  \item for $K=K_{\theta}$, there are two fixed rays, one of which is locally repelling and one that is neutral. Further, the neutral fixed ray is repelling on one side and attracting on the other;
  \item for $K>K_{\theta}$, there are three fixed rays, one of which is locally attracting and two that are locally repelling.
\end{itemize}
\end{corollary}

In particular, the value of $c$ plays no role in how many fixed external rays $f$ has. Theorem \ref{s2t2} also has the following immediate corollary.

\begin{corollary}
With the notation as above, if $f$ has one fixed external ray $E_\phi$ then $\{f^{-k}(E_\phi)\}_{k=0}^{\infty}$ is dense in a neighbourhood of infinity. If $f$ has two or three fixed external rays, then the basin of attraction of the non-repelling fixed external ray is dense in a neighbourhood of infinity.
\end{corollary}

We can use Theorem \ref{s2t2} to give a complete decomposition of the plane into dynamically important sets for $H$, that is, in the case $c=0$.
For a quasiregular mapping of polynomial type whose degree is larger than the distortion, it was proved in \cite{FN} that the escaping set is a connected neighbourhood of infinity. However, such mappings can have dynamically undesired behaviour outside the closure of the escaping set, for example in \cite{Bergweiler} a mapping is constructed which locally behaves like a winding mapping. We show that this does not happen for $H$. Recall that the escaping set of $H$ is given by
\[ I(H)=\{z\in\C\;|\;|H^n(z)|\to\infty \mbox{ as } n\to\infty\}.\]

\begin{corollary}\label{s2c1}
Let $K>1$, $\theta \in (-\pi/2, \pi/2]$ and $H(z) = h_{K,\theta}(z)^2$. Then $\C = I(H) \cup \partial I(H) \cup \mathcal{A}(0)$, where
$\mathcal{A}(0)$ is the basin of attraction of the fixed point $0$.
\end{corollary}

The next result is a refinement of a result from \cite{FF}. We first make the following definition.

\begin{definition}
Given a plane domain $U$, a quasiregular mapping $f:U \to \C$ is called \emph{nowhere uniformly quasiregular} if for every open set $V \subset U$, the maximal dilatation of $f^n \av _{V}$ is unbounded.
\end{definition}

For example, the quasiconformal mapping $f(x+iy) = Kx+iy$ is easily seen to be nowhere uniformly quasiregular for any $K>1$.

\begin{theorem}\label{s2t3}
Let $K>1$ and $\theta \in (-\pi/2, \pi/2]$. Then the mapping $h_{K,\theta}(z)^2 + c$ is nowhere uniformly quasiregular.
\end{theorem}

This theorem implies a result of \cite{FF}, that $h(z)^2+c$ cannot be quasiconformally conjugate to $z^2+c'$ in a neighbourhood of infinity.
We will next see that $h_{K_1,\theta_1}(z)^2+c_1$ is not quasiconformally conjugate to $h_{K_2,\theta_2}(z)^2+c_2$ for most pairs
$(K_1,\theta_1) \neq (K_2,\theta_2)$ in $(1,\infty) \times (-\pi/2, \pi/2]$. We conjecture that this is true for every such pair $(K_1,\theta_1) \neq (K_2,\theta_2)$, but our methods do not show this.

To this end, let $R_{\phi}$ be a fixed ray of $H$.
We recall from \cite{FF} that the complex dilatation of $H^n$ at $z\in R_{\phi}$ is given by
$\mu_n(z)=A^{n-1}(\mu_1)$ where $A$ is the M\"{o}bius transformation
\[ A(w) = \frac{ \mu_1 + e^{-i\phi}w} {1+e^{-i\phi}\overline{\mu_1} w}\]
and $\mu_1=e^{2i\theta}(K-1)/(K+1)\in\mathbb{D}$. The trace of $A$, denoted $\tr(A)$, is the trace of the matrix representing $A$, which is normalized to have determinant $1$.

Given $K_1,K_2 >1$ and $\theta_1,\theta_2 \in (-\pi/2, \pi/2]$, let $H_1,H_2$ be $h_{K_1,\theta_1}^2, h_{K_2,\theta_2}^2$ respectively.
Denote the fixed rays of $H_1$ by $R_{\phi_i}$ and the fixed rays of $H_2$ by $R_{\psi_j}$, the corresponding M\"{o}bius transformations of each fixed ray $R_{\phi_i}$ by $A_i(z)$ and the corresponding M\"{o}bius transformations of each fixed ray $R_{\psi_j}$ by $B_j(z)$.
We prove the following.

\begin{theorem}\label{s2t4}
With the notation above, there is no quasiconformal conjugacy between $H_1$ and $H_2$ in any neighbourhood of infinity if any of the following conditions hold:
\begin{enumerate}
\item the mappings $H_1,H_2$ have different numbers of fixed rays;
\item $H_1$ and $H_2$ both have one fixed ray, $R_{\phi_1}$ and $R_{\psi_1}$ respectively, and $\tr(A_1)^2\neq \tr(B_1)^2$;
\item if $H_1$ and $H_2$ both have two fixed rays $R_{\phi_i}$ and $R_{\psi_i}$ for $i= 1,2$, where $\phi_1>\phi_2$ and $\psi_1>\psi_2$, and $\tr(A_i)^2 \neq \tr(B_i)^2$ for some $i$;
\item if $H_1$ and $H_2$ both have three fixed rays $R_{\phi_i}$ and $R_{\psi_j}$, $i,j\in\{0,1,2\}$ respectively, where $\phi_1>\phi_0>\phi_2$ and $\psi_1>\psi_0>\psi_2$, and $\tr(A_i)^2\neq \tr(B_i)^2$ for some $i$.
\end{enumerate}
\end{theorem}

We can rule out some further cases in the following corollary.

\begin{corollary}\label{s2ct4}
If $\theta\in[0,\pi/2)$ is fixed and $K_1,K_2>1$ with $K_1 \neq K_2$, then $H_{K_1,\theta}$ and $H_{K,\theta_2}$ are not quasiconformally equivalent on any neighbourhood of infinity.
\end{corollary}

The paper is organized as follows. In $\S3$ we gather some preliminary material including stating some results from the iteration theory of M\"{o}bius maps, the proofs of which are postponed to $\S9$. In $\S4$ we study the ray structure of the mapping $H$ and prove Theorem~\ref{s2t1}. In $\S5$ we prove Theorem~\ref{s2t2} and use it in $\S6$ to prove Corollary \ref{s2c1} and in $\S7$ to prove Theorem~\ref{s2t3}. Then in $\S8$ we prove Theorem~\ref{s2t4}. Finally $\S9$ is an appendix on some results on M\"{o}bius maps.

The authors wish to thank Vladimir Markovic and Sebastian van Strien for enlightening conversations.

\section{Preliminaries}

\subsection{Quasiregular mappings}

A \emph{quasiconformal} mapping $f:\C \to \C$ is a homeomorphism such that $f$ is in the Sobolev space $W^1_{2, loc}(\C)$ and there exists $k\in[0,1)$ such that
the \emph{complex dilatation} $\mu_f = f_{\zbar}/ f_z$ satisfies
\[ | \mu_f(z) | <k\]
almost everywhere in $\C$. See for example \cite{FM} for more details on quasiconformal mappings.
The \emph{distortion} of $f$ at $z\in\C$ is
\[K_f(z):=\frac{1+|\mu_f(z)|}{1-|\mu_f(z)|}.\]
A mapping is called $K$-quasiconformal if $K_f(z) \leq K$ almost everywhere. The smallest such constant is called the \emph{maximal dilatation} and denoted by $K_f$.
If we drop the assumption on injectivity, then $f$ is a \emph{quasiregular mapping}. See for example \cite{IM,Rickman} for the theory of quasiregular mappings.
In the plane, every quasiregular mapping has a useful decomposition.

\begin{theorem}[Stoilow factorization, see for example \cite{IM} p.254]
\label{Stoilow}
Let $f:\C \rightarrow \C$ be a quasiregular mapping. Then there exists an analytic function $g$ and a quasiconformal mapping $h$ such that $f = g \circ h$.
\end{theorem}

We call a quasiregular mapping $f$ \emph{uniformly quasiregular} if there exists $K\geq1$ such that
$K_{f^n}(z)\leq K$ for all $n\in\N$.
The dynamics of uniformly quasiregular mappings in the plane are well understood due to the following result of Hinkkanen.

\begin{theorem}[\cite{Hinkkanen}]
Every uniformly quasiregular map $f:\C\to\C$ is quasiconformally conjugate to a holomorphic map.
\end{theorem}

This means that the dynamics of uniformly quasiregular maps of the plane reduce to complex dynamics, which gives us no new features.
To ensure our study is of independent interest, we require that mappings of the form $h(z)^2+c$ are not uniformly quasiregular.

\subsection{Definition of $H$}

Recall that if $K > 1, \theta \in (-\pi/2, \pi/2]$, then $h_{K,\theta}$ is the affine mapping given by
\[ h_{K,\theta} = \left ( \frac{K+1}{2} \right ) z + e^{2i\theta} \left ( \frac{K-1}{2} \right ) \overline{z}.\]
Using the formula for complex dilatation, we see that
\begin{equation}
\label{eq3.3}
\mu _{h_{K,\theta}} = e^{2i\theta} \frac{K-1}{K+1},
\end{equation}
and so $\av \av \mu_{h_{K,\theta}} \av \av _{\infty} <1$ which verifies that $h_{K,\theta}$ is quasiconformal with constant complex dilatation.
We now prove Proposition \ref{canonicalform} that states every quasiregular mapping in the plane of degree two with constant complex dilatation is linearly conjugate to a unique mapping of the form $h_{K,\theta}(z)^2+c$ for some $K > 1, \theta \in (-\pi/2, \pi/2]$ and $c \in \C$.

\begin{proof}[Proof of Proposition \ref{canonicalform}]
Let $f$ satisfy the hypotheses of the proposition and let $\mu_f \equiv \mu$. By Theorem \ref{Stoilow}, we can write $f=\widetilde{g} \circ \widetilde{h}$ for some quadratic polynomial $\widetilde{g}$ and quasiconformal map $\widetilde{h}$ with constant complex dilatation. We may assume that $\widetilde{h}$ fixes $0$.

Let $h=h_{K,\theta}$, where $K,\theta$ are chosen such that $\left ( \frac{K-1}{K+1} \right ) e^{2i\theta} = \mu$. Then by the formula for the complex dilatation of a composition (see, for example, \cite{FM}), we have
\[ \mu_{ \widetilde{h} \circ h_{K,\theta}^{-1}} \equiv 0.\]
Therefore, there exists a conformal map $A:\C \to \C$ such that $\widetilde{h} = A \circ h_{K, \theta}$.

We can therefore write $f=g \circ h$, where $g = \widetilde{g} \circ A$ is a quadratic polynomial and $h=h_{K,\theta}$.
Finally, applying \cite[Proposition 3.1]{FG} gives the result.
\end{proof}

In this paper we focus on the case where $c=0$ and we suppress the subscripts $K$ and $\theta$ where there will be no confusion.
We can restrict ourselves to studying only the $c=0$ case because of the following theorem from \cite{FF}, which gives an analogue of B\"{o}ttcher coordinates for these mappings.

\begin{theorem}{\cite[Theorem 2.1]{FF}}
\label{thmBot}
Let $K>1, \theta \in (-\pi/2,\pi/2]$, $h=h_{K,\theta}$ be an affine mapping and $c \in \C$. Then there exists a neighbourhood $U=U(K,\theta,c)$ of infinity and a quasiconformal map
$\psi = \psi(K,\theta,c)$ such that
\[H(\psi(z)) = \psi(f(z)),\]
for $z \in U$, where $f(z) = h(z)^2+c$.
\end{theorem}

From Proposition \ref{canonicalform} we know that any degree two mapping of constant complex dilatation is linearly conjugate to a mapping $f_{K,\theta,c}$ for some $K,\theta,c$. Then Theorem~\ref{thmBot} tells us that $f_{K,\theta,c}$ is quasiconformally conjugate to $h_{K,\theta}^2$ in a neighbourhood of infinity. Therefore we may restrict our attention to the study of dynamics of the mappings $h_{K,\theta}^2$.
We also note that for a fixed $K$ the maps $h_{K,\theta}^2$ and $h_{K,-\theta}^2$ are related.

\begin{lemma}\label{Upsilon}
We have $H_{K,-\theta}(z)=\overline{H_{K,\theta}(\overline{z})}$.
\end{lemma}
\begin{proof}
\begin{align*}
\overline{H_{K,\theta}(\overline{z})}
&=\left ( \frac{K+1}{2} \right ) z + e^{-2i \theta}\left ( \frac{K-1}{2} \right )\zbar\\
&=H_{K,-\theta}(z).
\end{align*}
\end{proof}

Lemma~\ref{Upsilon} means that we can just study the range $\theta\in[0,\pi/2]$ then transfer the results using complex conjugation to extend to $\theta\in(-\pi/2,0)$.

Fix $K>1, \theta \in (-\pi/2, \pi/2]$ and $H=h_{K,\theta}^2$. Contained in the proof of \cite[Theorem 6.4]{FG} is the observation that $H$ takes the form
\begin{equation}
\label{hpolar}
H(re^{i\varphi}) = r^2(1+(K^2-1)\cos^2(\varphi-\theta)) \exp \left [ 2i \left( \theta + \tan^{-1} \left ( \frac{\tan(\varphi-\theta)}{K} \right ) \right ) \right ]
\end{equation}
in polar coordinates. Hence $H$ maps rays to rays. First observed in \cite{FG} is the fact that $H$ has at least one fixed ray and that fixed rays correspond to roots of the cubic polynomial
\begin{equation}\label{cubic}
P(t):=Kt^3 + (2-K)\tan(\theta/2)t^2 + (2-K)t + K\tan(\theta/2),
\end{equation}
where $t=\tan[(\varphi-\theta)/2]$. It is shown that $P$ always has a root $t_0\in(-1,1)$ which corresponds to a fixed ray with angle in $(-\pi/2, \pi/2)$, so there is always at least one fixed ray. The fact that $P$ is a cubic suggests there could be one, two or three fixed rays and we will show that all cases are possible.

In \cite[Proposition 7.3]{FF} the following proposition is proved.

\begin{proposition}\label{s6l1}
Let $R_\phi$ be a fixed ray of $H$. Then the complex dilatation $\mu_{H^n}(z)$ of $H^n$ for $z\in R_\phi$ does not depend on $z$ and satisfies
$\mu_{H^n}=A^{n-1}(\mu_H)$ where $A$ is the M\"{o}bius map
\begin{equation}\label{s6eA}
A(w) = \frac{ \mu_H + e^{-i\phi}w} {1+e^{-i\phi}\overline{\mu_H} w}.
\end{equation}
In particular,
\[|\mu_{H^n}(z)|\to 1 \mbox{ as } n\to\infty.\]
\end{proposition}

\subsection{M\"{o}bius maps and Blaschke products}

To prove Theorem~\ref{s2t4} we will need some results on hyperbolic M\"{o}bius maps of $\D$, which we will state now and prove in $\S9$.
Recall that a M\"{o}bius map $A:\D\to\D$ is called \emph{hyperbolic} if $\tr^2(A)>4$, where $\tr$ denotes the trace of the normalized matrix representing $A$. The connection here is that in \cite{FF} it was proved that the M\"{o}bius map given in \eqref{s6eA} is always hyperbolic.

\begin{lemma}\label{sMl1}
Let $A:\D\to\D$ be a hyperbolic M\"{o}bius map. Then $A$ is conjugate to a M\"{o}bius map $\widehat{A}:\Hp \to \Hp$ given by
\[\widehat{A}(z)=kz\]
where
\[k=(T-2-(T^2-4T)^{\frac{1}{2}})/2 < 1,\]
and $T:=\tr^2(A).$
\end{lemma}

The following lemma on sequences of hyperbolic M\"{o}bius transformations is a combination of results from \cite{Gill} and \cite{MM} .

\begin{theorem}\label{sMlMM}\cite{Gill,MM}
Let $A,A_j$ be hyperbolic M\"{o}bius maps of $\D$ such that $A^n(z)\to\alpha\in\partial\D$ as $n\to\infty$ and $A_j \to A$ locally uniformly as $j\to\infty$.
Suppose we have sequences $t_n, s_n$ of hyperbolic M\"{o}bius maps of $\D$ defined by
\begin{align*}
t_n(z)&=A_1\circ A_2\circ\ldots\circ A_n(z),\\
s_n(z)&=A_n\circ A_{n-1}\circ\ldots\circ A_1(z).
\end{align*}
Then both $t_n(z) \to \alpha$ and $s_n(z)\to \alpha$ as $n\to \infty$ for all $z\in\D$.
\end{theorem}

We will use this result to prove the following theorem, the proof of which is postponed until $\S8$.

\begin{theorem}\label{sMl3}
Let $A,A_j:\D\to\D$ be hyperbolic M\"{o}bius maps such that $A^n(z)\to\alpha\in\partial\D$ as $n\to\infty$ and $A_j\to A$ locally uniformly as $j\to\infty$. Let
\[t_n(z)=A_1\circ A_2\circ\ldots\circ A_n(z).\]
Then
\[d_h(0,t_n(z))=\log\left[\frac{1}{\prod_{j=1}^n k_j}\right] + O(1),\]
for large $n$, where $d_h$ denotes the hyperbolic metric on $\D$, $k_j<1$ for all $j$ and $k_j\to k$, where $k_j,k$ are the quantities defined in Lemma~\ref{sMl1} associated to $A_j,A$.
\end{theorem}

In particular, if $A_j = A$ for every $j \in \N$, then
\[d_h(0,A^n(z))=\log\left[1/k^n\right] + O(1),\]
as $n \to \infty$.

Will we need to use the following result on Blaschke products, see for example \cite{Beardon,CG}.
A Blaschke product $B$ is given by
\[B(z):=\zeta\prod_{i=1}^n\left({{z-a_i}\over {1-\overline{a_i}z}}\right)^{m_i},\]
where $\zeta\in\partial\D$ and $|a_i|<1$.

We are only concerned with Blaschke products of degree two and in this case we have the following standard result, see for example \cite{CG}.

\begin{proposition}\label{Blaschke}
Let $B$ be a Blaschke product of degree $2$. Then the Julia set $J(B)$ is contained in $S^1$ and we have the following cases:
\begin{itemize}
\item If $B$ has one fixed point in $S^1$, one fixed point in $\D$ and one fixed point in $\C \setminus \overline{\D}$, then $J(B) = S^1$.
\item If $B$ has one fixed point in $S^1$ of multiplicity three, and no other fixed points, then $J(B) = S^1$.
\item If $B$ has one repelling and one neutral fixed point in $S^1$, then $J(B)$ is a Cantor subset of $S^1$.
\item If $B$ has three fixed points in $S^1$, then $J(B)$ is a Cantor subset of $S^1$.
\end{itemize}
\end{proposition}

\section{Fixed rays of $H$}

To prove Theorem~\ref{s2t1} we use the following strategy.
\begin{itemize}
  \item Given $K,\theta$ show that the argument of $H=H_{K,\theta}$ induces a map $\widetilde{H}:S^1 \to S^1$.
  \item Determine the possible locations of fixed points of $\widetilde{H}$.
  \item When $\theta = 0$, show that if $K \leq 2$ then $\widetilde{H}$ has one repelling fixed point, or neutral in the case $K=2$, and if $K>2$ then $\widetilde{H}$ has three fixed points, two repelling and one attracting.
  \item When $\theta = \pi/2$, show that $\widetilde{H}$ only ever has one repelling fixed point.
  \item For $\theta \in (0,\pi/2)$, show that there exists $K_{\theta} >2$ such that $\widetilde{H}$ has two fixed points, one repelling and one neutral. If $K<K_{\theta}$, then $\widetilde{H}$ has one repelling fixed point. If $K>K_{\theta}$, then $\widetilde{H}$ has three fixed points, one attracting and two repelling.
\end{itemize}

\subsection{The induced map $\widetilde{H}$ of $S^1$.}

Given $K,\theta$ the map $H=H_{K,\theta}$ induces a map $S^1 \to S^1$ as follows.

\begin{definition}
Define $\widetilde{H}:S^1\to S^1$ by:
\[\widetilde{H}(e^{i\varphi})=e^{i\psi} \mbox{ where } \arg[H(re^{i\varphi})]=\psi \]
for any $r>0$.
\end{definition}

By lifting $\widetilde{H}$ to $\R$, we obtain a $2\pi$-periodic mapping $\R \to \R$. We will often confuse $\widetilde{H}$ with its lift to $\R$, but the usage of $\widetilde{H}$ should be clear from the context. We also remark that by the definition of $H$, $\widetilde{H}$ is actually a $\pi$-periodic mapping. We have that
\begin{equation}\label{polar}
\widetilde{H}(\varphi)= 2\theta + 2\tan^{-1} \left ( \frac{\tan(\varphi-\theta)}{K} \right )
\end{equation}
when $\widetilde{H}$ is viewed as the mapping lifted to $\R$.

\begin{figure}[h]\
\begin{center}
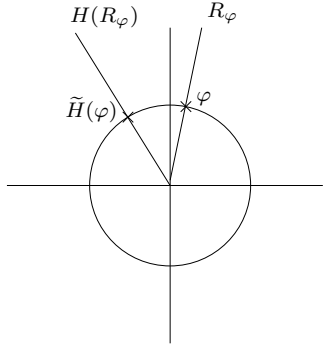
\caption{Diagram showing how $\widetilde{H}$ is induced from the action of $H$ on the ray $R_\varphi$.}\label{s4f4}
\end{center}
\end{figure}

Viewed as a mapping on $S^1$, $\widetilde{H}$ is two-to-one. Points in $S^1$ correspond to rays in $\C$ and so fixed points of $\widetilde{H}$ correspond to fixed rays of $H$. In this way, we reduce our study of fixed rays of $H$ to fixed points of the circle endomorphism $\widetilde{H}$.
Given a sector $S \subset \C$, we will denote by $\widetilde{S}$ the corresponding subset of $S^1$ or interval in $\R / 2\pi \Z$.

We also define the homeomorphism $\widetilde{h}$ on $S^1$ induced by $h=h_{K,\theta}$.

\begin{definition}\label{h}
Define the map $\widetilde{h}:\R\to \R$ by:
\[\widetilde{h}(\varphi)=\psi \mbox{ where } \arg[h(re^{i\varphi})]= \psi\]
for any $r>0$.
\end{definition}

Recall from Lemma~\ref{Upsilon} that
\[H_{K,-\theta}(z)=\overline{H_{K,\theta}(\overline{z})}.\]
Hence we only need to consider the case $\theta\in[0,\pi/2]$. The ray $R_\varphi$ is a fixed ray of $H_{K,\theta}$ if and only if $R_{-\varphi}$ is a fixed ray of $H_{K,-\theta}$ and they have the same behaviour. For the rest of this section we assume $\theta\in[0,\pi/2]$.

\subsection{Locations of fixed rays of $H$}

By results in \cite{FG}, we know that $H$ has at least one fixed ray. First, we will narrow down the sectors where these fixed rays can be.

\begin{lemma}
\label{s4l1}
If $\theta > 0$ then any fixed ray $R_{\phi}$ of $H$ lies in one of the sectors
\[ \mathbb{F}^+_\theta=\{R_\varphi \;|\; 2\theta < \varphi < \theta + \pi/2\}, \]
or
\[ \mathbb{F}^-_\theta=\{R_\varphi \;|\; \theta-\pi/2 < \varphi < 0\}. \]
If $\theta=0$ any fixed ray is in $\mathbb{F}_\theta^\pm\cup\{R_0\}$. If $\theta=\pi/2$ then $R_0$ is the only fixed ray.
\end{lemma}

\begin{figure}[h]
\begin{center}
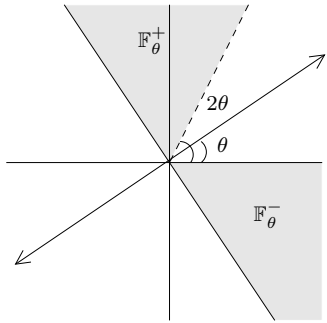
\caption{Diagram showing the regions $\mathbb{F}_\theta^\pm$.}\label{s4f3}
\end{center}
\end{figure}

\begin{proof}
Recall that our normalization for $\theta$ requires $\theta \in (-\pi/2,\pi/2]$ and that by Lemma \ref{Upsilon} we need only consider $\theta\geq 0$. Define the following quadrants of $\C$:
\begin{align*}
Q_1 &= \{R_\varphi \;|\; 0 < \varphi-\theta < \pi/2\},\\
Q_2 &= \{R_\varphi \;|\; -\pi/2 < \varphi-\theta <0\},\\
Q_3 &= \{R_\varphi \;|\; -\pi/2 < \varphi-\theta < -\pi\}, \\
Q_4 &= \{R_\varphi \;|\; \pi/2 < \varphi-\theta < \pi\}.
\end{align*}

Let $0< \theta <\pi/2$. First we consider $Q_4$. Under $H$, the image of $Q_4$ is
\[ H(Q_4)=\{R_\varphi \;|\; -\pi<\varphi-2\theta<0\}. \]
We notice that $Q_4\cap H(Q_4)=\emptyset$ and so there can be no fixed
ray in the sector $Q_4$. Next we consider the rays in $Q_3$. For simplicity we will consider rays to have angle between $-2\pi$ and $0$. Now
\[ H(Q_3)=\{R_\varphi \;|\; -2\pi< \varphi-2\theta < \pi\}. \]
Recalling that $0<\theta<\pi/2$, we have $H(Q_3)\cap Q_3 \neq
\emptyset$, so it is possible that there is a fixed ray in $Q_3$.
However note that $h(Q_3)=Q_3$ and that for $R_\varphi \in Q_3$ if
$R_\psi=h(R_\varphi)$ then $-\pi<\psi<\varphi<0$. Squaring doubles the
angle so if $R_\tau=H(R_\varphi)$ the angles must satisfy
\[ -2\pi<\tau<\psi<\varphi<0. \]
This holds for all $R_\varphi\in Q_3$ and so there can be no fixed ray in $Q_3$.

Next consider the quadrant $Q_1$. Note that $H(R_\theta)=R_{2\theta}$, so $H(Q_1)\cap Q_1 = \mathbb{F}^+_\theta$. Hence any fixed ray of $Q_1$ must lie in $\mathbb{F}^+_\theta$. Finally, if $R_\varphi\in Q_2$ and $\varphi>0$ then, similar to the case of rays in $Q_3$, if $R_\psi=h(R_\varphi)$ then $0<\varphi<\psi<\theta$. Squaring further increases the angle so if $R_\tau=H(R_\varphi)$ then the angles must satisfy
\[0<\varphi<\psi<\tau<2\theta.\]
Hence any fixed ray of $Q_2$ must lie in $\mathbb{F}^-_\theta$ as claimed.

When $\theta=0$ the above holds with the addition that the ray $R_0$ is always fixed. It is easy to see that $R_0$ is fixed when $\theta = \pi/2$, and similar arguments to the above show that this is the only possible fixed ray.
\end{proof}

\subsection{Local expansion and contraction}

In this subsection we will study $\widetilde{H}'$, as this determines whether a small neighbourhood of $\varphi$ is contracted or expanded under $\widetilde{H}$. Since $\widetilde{H}$ is sense-preserving, $\widetilde{H}'>0$. Let us make this more precise.

\begin{definition}\label{expandcontract}
An interval $I\subset \R/2\pi\Z$ is \emph{expanded} by $\widetilde{H}$ if
\[|\widetilde{H}(I)|>|I|,\]
and is \emph{contracted} by $\widetilde{H}$ if
\[|\widetilde{H}(I)|<|I|.\]
\end{definition}

It is easy to see that if $\widetilde{H}'(\varphi)<1$ or $\widetilde{H}'(\varphi)>1$ then there exists some neighbourhood $V$ of $\varphi$ that is contracted or expanded respectively by $\widetilde{H}$. Further if there is some closed interval $I$ such that $\widetilde{H}'(\varphi)<1$ or $\widetilde{H}'(\varphi)>1$ for all $\varphi\in I$ then it follows that $I$ is contracted or expanded respectively by $\widetilde{H}$.

\begin{lemma}\label{s4l3}
For $K<2$ and any $\theta$, $\widetilde{H}'(\varphi)>1$ for all $\varphi \in \R/2\pi \Z$. When $K>2$ there is a single interval $J\subset(\theta-\pi/2,\theta+\pi/2)$ where we have $\widetilde{H}'(\varphi)<1$ and further, $\theta$ is the midpoint of $J$.
\end{lemma}

\begin{proof}
By differentiating the expression for $\widetilde{H}$ we obtain
\begin{equation}\label{diff}
\widetilde{H}'(\varphi)=\frac{2K}{1 + (K^2 - 1)\cos^2(\varphi - \theta)}.
\end{equation}
Note that $\widetilde{H}'$ is continuous and that
\begin{equation}\label{s4e01}
\widetilde{H}'(\varphi)<1 \iff 2K<1+(K^2-1)\cos^2(\varphi-\theta).
\end{equation}
It is easy to see that
\[1+(K^2-1)\cos(\varphi-\theta)\leq K^2,\]
and hence if $K<2$
\begin{equation}\label{s4e02}
2K>K^2\geq1+(K^2-1)\cos^2(\varphi-\theta).
\end{equation}
Then \eqref{s4e01} and \eqref{s4e02} imply $\widetilde{H}'(\varphi)>1$ when $K<2$ proving the first part of the lemma.
Considering $\varphi=\theta$, we see
\begin{equation}\label{s4e03}
\widetilde{H}'(\theta)=\frac{2K}{1 + (K^2 - 1)\cos^2(0)}=\frac{2}{K}.
\end{equation}
It is easy to see that if $K=2$ then $\widetilde{H}'(\theta)=1$ and for $\varphi\neq\theta$ we have $\widetilde{H}'(\varphi)>1$.
For $K>2$ we have $\widetilde{H}'(\theta)<1$ by \eqref{s4e03}. As $\widetilde{H}'$ is continuous we must have some interval $J$ containing $\theta$ such that $\widetilde{H}'(\varphi)<1$ for $\varphi\in J$. We want to show this is the only interval of $(\theta-\pi/2,\theta+\pi/2)$ with this property. Note that
\[\widetilde{H}'(\varphi)\to 2K \mbox{ as } \varphi-\theta\to\pm\pi/2,\]
so $J\neq(\theta-\pi/2,\theta+\pi/2).$ We have to show there is no other region where $\widetilde{H}'<1$. To do this we differentiate again to obtain
\[\widetilde{H}''(\varphi)=\frac{(2K^3-2K)\tan(\varphi-\theta)}{\cos^2(\varphi-\theta)(K^2+\tan^2(\varphi-\theta))^2}.\]
It is easy to see that when $\widetilde{H}''(\varphi)=0$, that we are at a local minimum or maximum of $\widetilde{H}'$. Now $\widetilde{H}''(\varphi)=0$ implies
\[(2K^3-2K)\tan(\varphi-\theta)=0\]
which, as $K>2$ and $\varphi-\theta\in(-\pi/2,\pi/2)$, implies that $\varphi=\theta.$

We know that $\widetilde{H}'(\varphi)>1$ for $\varphi$ near $\pm\pi/2$ and that there is only one critical point of $\widetilde{H}$ at $\varphi=\theta$. Hence $J$ is the only interval such that $\varphi\in(\theta-\pi/2,\theta+\pi/2)$ implies $\widetilde{H}'(\varphi)<1$. The final statement that $\theta$ is the midpoint of $J$ follows from the fact that $\cos^2(\varphi-\theta)$, and so $\widetilde{H}'$ too, is symmetric about $\theta$.
\end{proof}

\begin{definition}\label{jdef}
Given $K>2$ and $\theta$, denote by $J=J_K$ the interval $(\theta - \eta, \theta + \eta)$, for $\eta = \eta_K$, where $\widetilde{H}'(\theta)<1$.
Note that $\eta$ does not depend on $\theta$.
\end{definition}

We remark that as $\widetilde{H}$ is $\pi$-periodic, the translate of $J$ by $\pi$
is a second interval where $\widetilde{H}'(\varphi)<1$. However, there can be no fixed points here from Lemma~\ref{s4l1} and so we are not concerned with this other interval in this section.

\subsection{Special cases}

We will now investigate the fixed points of $\widetilde{H}$. The cases where $\theta=0$ and $\theta=\pi/2$ are special cases. We first show that if $\theta = 0$, then $\widetilde{H}$ can never have a neutral fixed point which is not $0$.

\begin{lemma}\label{H'0}
Let $\theta=0$. Suppose $\varphi\neq0$ and $\widetilde{H}'(\varphi)=1$, then $\varphi$ cannot be fixed.
\end{lemma}

\begin{proof}
As $\widetilde{H}'(\varphi)=1$ then \eqref{diff} implies
\[\varphi=\cos^{-1}\left[\left(\frac{2K-1}{K^2-1}\right)^\frac{1}{2}\right],\]
where we take the positive square root since $|\varphi|<\pi/2$.
Suppose that $\varphi$ is fixed so that $\widetilde{H}(\varphi) = \varphi$, then \eqref{polar} implies
\begin{equation}\label{eH'1}
\cos^{-1}\left[\left(\frac{2K-1}{K^2-1}\right)^\frac{1}{2}\right]=2\tan^{-1}\left[\frac{\tan\left(\cos^{-1}\left[\left(\frac{2K-1}{K^2-1}\right)^\frac{1}{2}\right]\right)}{K}\right].
\end{equation}
Applying $\cos$ to both sides of \eqref{eH'1}, using the double angle formula for $\cos$ and the identity $\cos^2 (\tan^{-1} x) = (1+x^2)^{-1}$, we obtain
\begin{align}
\left(\frac{2K-1}{K^2-1}\right)^\frac{1}{2}&=\frac{1-\tan^2\left(\cos^{-1}\left[\left(\frac{2K-1}{K^2-1}\right)^\frac{1}{2}\right]\right)/K^2}{1+\tan^2\left(\cos^{-1}\left[\left(\frac{2K-1}{K^2-1}\right)^\frac{1}{2}\right]\right)/K^2}\notag\\
&=\frac{K^2-\tan^2\left(\cos^{-1}\left[\left(\frac{2K-1}{K^2-1}\right)^\frac{1}{2}\right]\right)}{K^2+\tan^2\left(\cos^{-1}\left[\left(\frac{2K-1}{K^2-1}\right)^\frac{1}{2}\right]\right)}.\label{eH'2}
\end{align}
Using the formula $\tan^2(\cos^{-1}X)=(1-X^2)/X^2$ and rearranging, \eqref{eH'2} becomes
\[(2K-1)^{\frac{1}{2}}(K^2(2K-1)+(K^2-1)-(2K-1))=(K^2-1)^{\frac{1}{2}}(K^2(2K-1)-(K^2-1)+(2K-1)).\]
Rearranging and squaring both sides we see
\begin{equation}\label{eH'3}
4K^2(2K-1)(K^2-1)^2=4K^2(K^2-1)(K^2-K+1)^2.
\end{equation}
Hence $K=0$ and $K=1$ are solutions to \eqref{eH'1}. Factoring these solutions out of \eqref{eH'3} and expanding, we obtain
\begin{equation}\label{eH'4}
K^2(K-2)^2=0
\end{equation}
$K=0$ and $K=2$ are solutions of \eqref{eH'4}. Hence all possible solutions to \eqref{eH'1} are $K=0,1,2$. Since $K=0$ and $K=1$ are not permissible values, the only valid solution for us is $K=2$, which implies that $\varphi=0$. This completes the proof.
\end{proof}

\begin{lemma}\label{theta0}
If $\theta=0$ then $\widetilde{H}$ has one repelling fixed point $\phi_0=0$ when $K<2$, has one neutral fixed point when $K=2$ and has three fixed points $\-\pi/2<\phi_2<\phi_0=0<\phi_1<\pi/2$ when $K>2$, where $\phi_1$ and $\phi_2$ are repelling and $\phi_0$ is attracting. Further $\phi_2=-\phi_1$.
\end{lemma}

\begin{figure}[h]
\begin{center}
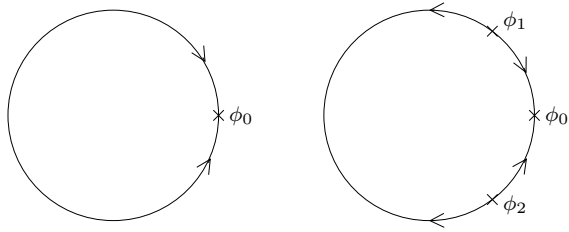
\caption{Diagram showing the local dynamics of $\phi_i$ in the two cases.}\label{s4f5}
\end{center}
\end{figure}

\begin{proof}
First substitute $\theta=0$ into \eqref{polar} to obtain
\[\widetilde{H}(\varphi)=2\tan^{-1}[(\tan\varphi)/K].\]
Then any fixed point $\phi$ must satisfy the equation
\begin{equation}\label{eqzero}
K\tan(\phi/2)=\tan\phi.
\end{equation}
Since $\phi_0=0$ satisfies \eqref{eqzero}, it is always a fixed point when $\theta=0$. Lemma~\ref{s4l3} implies that for $K<2$, $\widetilde{H}'(\phi_0)>1$, so $\phi_0$ is repelling. When $K>2$ we see $\widetilde{H}'(0)<1$ and $\phi_0=0$ is attracting. Let $K<2$ and suppose we had some other fixed point $\phi$. Without loss of generality, assume $\phi>\phi_0$ then the interval $[\phi_0,\phi]$ is fixed under $\widetilde{H}$, since the image of $[0,\pi]$ under $\widetilde{H}$ is $[0,2\pi]$. However this interval must also be expanded, as all $\varphi\in[\phi_0,\phi]$ satisfy $\widetilde{H}'(\varphi)>1$, a contradiction.

Now let $K>2$. Recall Definition \ref{jdef} and the interval $J$. Write $J=(-\eta,\eta)$ and recall that Lemma~\ref{H'0} implies that neither $\pm \eta$ can be fixed for $K>2$, and hence
\[|\widetilde{H}(J)|<|J|.\]
Since $\widetilde{H}(\pi/2) > \pi/2$ and $\widetilde{H}(\eta) < \eta$, by continuity there exists a fixed point $\phi_1 \in (\eta, \pi/2)$. Similarly, there is a fixed point $\phi_2 \in (-\pi/2, -\eta)$.
Further, $\widetilde{H}'(\phi_i)>1$ for $i=1,2$, so they are repelling. As $\phi_1$ and $\phi_2$ must satisfy \eqref{eqzero} and since $\tan$ is odd we must have $\phi_2=-\phi_1$.
Since $\widetilde{H}$ can have at most three fixed points in $S^1$, these account for them all.

Finally we deal with the case when $K=2$. Here we have from Lemma~\ref{s4l3} that $\widetilde{H}'(0)=1$ and so $\phi_0$ is a neutral fixed point. However $\widetilde{H}'(\varphi)>1$ for $\varphi\in(-\pi/2,\pi/2) \setminus \{ 0 \}$, so any interval with one end-point $0$ is expanded by $\widetilde{H}$. This implies there are no other fixed points.
\end{proof}

\begin{lemma}
If $\theta=\pi/2$ then $\phi_0=0$ is the only fixed point of $\widetilde{H}$ for all $K>1$ and it is always repelling.
\end{lemma}

\begin{proof}
By Lemma \ref{s4l1} we know $\phi_0=0$ is the only fixed point of $\widetilde{H}$. Substituting $\varphi=0$ and $\theta=\pi/2$ into \eqref{diff} we see
\[\widetilde{H}'(\phi_0)=2K.\]
As $K>1$ we have that $\phi_0$ is repelling.
\end{proof}

\subsection{The general case $\theta \in (0,\pi/2)$}

Recall the sectors $F_{\theta}^{\pm}$ from Lemma \ref{s4l1}, and the corresponding intervals $\widetilde{F}^{\pm}_{\theta}$ in $S^1$.

\begin{lemma}\label{s4l5}
There is exactly one fixed point $\phi\in\widetilde{\mathbb{F}}^-_\theta$ of $\widetilde{H}$ for all $K>1$. Further, it is a repelling fixed point.
\end{lemma}

\begin{proof}
Recalling the notation of Lemma \ref{s4l1},
we know any fixed point in $\widetilde{Q}_2$ must lie in $\widetilde{\mathbb{F}}^-_\theta$. We also have that $\widetilde{Q}_2\subset \widetilde{H}(\widetilde{Q}_2)$. Recall that $\widetilde{H}$ is orientation preserving, injective when restricted to $\widetilde{Q}_2$, and continuous. Hence there must be a fixed point $\phi\in\widetilde{\mathbb{F}}^-_\theta$. We will see that $\phi$ is the only fixed point in $\widetilde{\mathbb{F}}^-_\theta$.

From Lemma~\ref{s4l3} and Definition \ref{jdef}, $J\subset Q_1\cup Q_2$ where $\widetilde{H}'(\varphi)<1$ for $\varphi\in J$ and also $\theta\in J$. If $K\leq 2$ then $J=\emptyset$ or $J=\{\theta\}$ and every interval with end-point $\phi$ is expanded by $\widetilde{H}$. Therefore $\widetilde{H}$ has no other fixed points.

Finally suppose $K>2$ and $J \neq \emptyset$. Then since $\phi<0$ is fixed and $\theta >0$, the interval $[\phi, \theta]$ is expanded by $\widetilde{H}$ and so $\phi \notin J$.
Suppose there is some other fixed point $\phi' \in\widetilde{\mathbb{F}}^-_\theta$. Without loss of generality, assume $\phi<\phi '$. Then the interval
$I=[\phi,\phi ']$
satisfies $\widetilde{H}(I)=I$ and $I\cap J=\emptyset$. Therefore $I$ is expanded by $\widetilde{H}$ which is a contradiction. Therefore there can only be one fixed point in $\widetilde{\mathbb{F}}^-_\theta$.
\end{proof}

In the next step, we will see that given $\theta \in (0,\pi/2)$, we can choose $K$ so that there are exactly two fixed points of $\widetilde{H}$. First we show a uniqueness lemma for neutral fixed points by similar calculations to Lemma \ref{H'0}.

\begin{lemma}\label{s4l001}
Let $\theta \in (0, \pi/2)$. There exists one $K_\theta>2$ such that $\widetilde{H}_{K_{\theta},\theta}$ has a neutral fixed point $\phi_{K_{\theta}}$.
\end{lemma}

\begin{proof}
Note that \eqref{diff} implies that if $\widetilde{H}_{K,\theta}'(\phi_{K_\theta})=1$ then
\[\phi_{K_\theta}=\cos^{-1}\left[\left(\frac{2K-1}{K^2-1}\right)^{\frac{1}{2}}\right]+\theta.\]
From \eqref{polar} and the assumption $\phi_{K_\theta}$ is fixed we must have
\begin{equation}\label{s4e111}
\cos^{-1}\left[\left(\frac{2K-1}{K^2-1}\right)^{\frac{1}{2}}\right]+\theta=2\tan^{-1}\left[\tan\left(\cos^{-1}\left[\left(\frac{2K-1}{K^2-1}\right)^{\frac{1}{2}}\right]\right)/K\right]+2\theta.
\end{equation}
Using the formula $\tan(\cos^{-1} X)=(1-X^2)^{\frac{1}{2}}/X$, rearranging, simplifying and applying $\tan$ to both sides; \eqref{s4e111} becomes
\begin{equation}\label{s4e112}
\tan\left[\left(\cos^{-1}\left[\left(\frac{2K-1}{K^2-1}\right)^{\frac{1}{2}}\right]-\theta\right)/2\right]=\left(\frac{K-2}{K(2K-1)}\right)^{\frac{1}{2}}.
\end{equation}
Applying the identity $\tan^2 x/2 = (1-\cos x )/(1+\cos x)$ and squaring both sides we see \eqref{s4e112} is equivalent to
\[\frac{1-\cos[\cos^{-1}([(2K-1)/(K^2-1)]^{\frac{1}{2}})-\theta]}{1+\cos[\cos^{-1}([(2K-1)/(K^2-1)]^{\frac{1}{2}})-\theta]}=\frac{K-2}{K(2K-1)}.\]
Applying the addition formula for $\cos$, the formula $\cos x=(1-\sin^2 x)^{\frac{1}{2}}$ and factoring out solutions of $K$ that are not permissible, we see that
\[(K^2-K+1)=(K^2-1)^{\frac{1}{2}}((2K-1)^{\frac{1}{2}}\cos\theta+K^{\frac{1}{2}}(K-2)^{\frac{1}{2}}\sin\theta)\]
We can rearrange and square to see that
\[((K^2-K+1)-(K^2-1)^{\frac{1}{2}}(2K-1)^{\frac{1}{2}}\cos\theta)^2 = K(K-2)(K^2-1)(1-\cos^2\theta).\]
Expanding and solving the quadratic in $\cos \theta$, we obtain
\begin{equation}\label{s4e114}
\theta=\cos^{-1}\left[\left(\frac{2K-1}{K^2-1}\right)^{\frac{3}{2}}(K-1)\right],
\end{equation}
since $\theta>0$.
Writing \eqref{s4e114} as $\theta=\cos^{-1}[f(K)]$, and viewing $f$ as a function $(2,\infty) \to \R$, one can calculate that
\[ f'(K) = -(2K-1)^{1/2}(K-1)^{-3/2}(K+1)^{-3/2}(K+4) <0\]
for $K>2$.
Hence $f$ is a decreasing function which converges to $0$ as $K \to \infty$. Hence $\cos^{-1}\circ f:(2,\infty)\to\R$ is an increasing function which is therefore injective. Further $f(2)=1$ and $f(K)>0$, hence
\[\cos^{-1}\circ f:(2,\infty)\to(0,\pi/2)\]
is bijective. By observing that given $\theta \in (0,\pi/2)$ we can find exactly one $K_{\theta}>2$ satisfying \eqref{s4e114}, this completes the proof.
\end{proof}

We next show that for this value $K_{\theta}$, the corresponding $\widetilde{H}$ has only the two fixed points constructed thus far.

\begin{lemma}\label{s4l7}
Let $\theta \in (0,\pi/2)$ and let $K=K_{\theta}$.
Then $\widetilde{H}$ has two fixed points, one of which is the neutral fixed point of Lemma \ref{s4l001}, $\phi_{K_\theta}\in\widetilde{\mathbb{F}}^+_\theta$, and one of which is the repelling fixed point $\phi\in\widetilde{\mathbb{F}}^-_\theta$.
\end{lemma}

\begin{proof}
Fix $\theta \in (0,\pi/2)$.
First, by Lemma \ref{s4l5}, there is always exactly one repelling fixed point of $\widetilde{H}$ in $\widetilde{\mathbb{F}}^-_\theta$, and so any remaining fixed points will lie in $\widetilde{\mathbb{F}}^+_\theta$.
We know from Lemma \ref{s4l001} that for a neutral fixed point we require $K>2$ for any $\theta \in (0,\pi/2)$. Recall the interval $J=J_K = (\theta - \eta_{K}, \theta + \eta_{K})$ from Definition \ref{jdef}, which is non-empty for $K>2$. Consider the subinterval
\[J^+_K=(\theta,\theta + \eta_K).\]
Writing $\varphi_K^+ = \theta + \eta_K$ and $\widetilde{H}_K$ to emphasize dependence on $K$, we note that $\widetilde{H}_K'(\varphi_K^+)=1$.
We want to show that some value $K=K_{\theta}$ will give us the neutral fixed point $\varphi_{K_{\theta}}^+$.

First, by \eqref{diff}, we have
\[\varphi_K^+=\cos^{-1}\left[\left(\frac{2K-1}{K^2-1}\right)^{\frac{1}{2}}\right]+\theta.\]
Next, \eqref{polar} implies that
\[\widetilde{H}_K(\varphi_K^+)=2\tan^{-1}\left(\left[\tan\left(\cos^{-1}\left[\left(\frac{2K-1}{K^2-1}\right)^{\frac{1}{2}}\right]\right)\right]/K\right)+2\theta.\]
Using the formula $\tan(\cos^{-1} x)=(1-x^2)^{1/2}/x$ and simplifying, we see that
\[\widetilde{H}_K(\varphi_K^+)=2\tan^{-1}\left[\left(\frac{K-2}{K(2K-1)}\right)^{\frac{1}{2}}\right]+2\theta.\]
For $K$ just above $2$, we see that $\widetilde{H}_K(\varphi_K^+) \approx 2\theta > \varphi_K^+$. Now as $K \to \infty$, we have both
$\varphi_K^+ \to \pi/2 + \theta$ and $\widetilde{H}_K(\varphi_K^+) \to 2 \theta$. Since $\theta < \pi/2$, we have $\widetilde{H}_K(\varphi_K^+) < \varphi_K^+$ for all large enough $K$.

By continuity, there exists some $K_{\theta}$ such that $\widetilde{H}_{K_{\theta}}(\varphi_{K_{\theta}}^+)~=~\varphi_{K_{\theta}}^+$.
By construction, $\widetilde{H}_{K_\theta}'(\varphi_{K_\theta}^+)=1$ and hence it is a neutral fixed point. By Lemma~\ref{s4l001} we know it is the only neutral fixed point for our given $\theta$. We take $\phi_{K_{\theta}}$ to be $\varphi_{K_{\theta}}^+$.

To see this is the only fixed point in $\widetilde{\mathbb{F}}^+_{\theta}$, consider any interval contained in $\widetilde{\mathbb{F}}^+_{\theta}$ with one endpoint at $\phi_{K_\theta}^+$. Then the interior of the interval is either contained in $J$ and the interval is contracted, or it is contained in the complement of $J$ and the interval is expanded. In either case, the other endpoint of the interval cannot be a fixed point.
\end{proof}

For $K<K_{\theta}$, the next lemma shows that we only have one fixed point.

\begin{lemma}\label{s4l8}
Let $\theta \in (0,\pi/2)$.
For $K<K_\theta$ there exists only one fixed point $\phi$ of $\widetilde{H}$. Further $\phi\in\widetilde{\mathbb{F}}^-_\theta$ and $\phi$ is repelling.
\end{lemma}

\begin{proof}
From Lemma \ref{s4l5} we know there must be exactly one repelling fixed point in $\widetilde{\mathbb{F}}^-_\theta$. By Lemma~\ref{s4l1}, we are left to show there are no fixed points in $\widetilde{\mathbb{F}}^+_\theta$.

Using the notation of the previous lemma, when $K<K_\theta$ we know that $\widetilde{H}_K(\varphi_K^+)>\varphi_K^+$.
Suppose we have a fixed point $\xi>0$. Then if $\xi<\varphi_K^+$, the interval $I=(\xi, \varphi_K^+)$ is contained in $J$ and $|\widetilde{H}(I)|<|I|$. However, $\widetilde{H}_K(\varphi_K^+)>\varphi_K^+$ which gives a contradiction.
On the other hand, suppose that $\xi > \varphi_K^+$. Then the interval $I'=(\varphi_K^+, \xi)$ similarly satisfies $|\widetilde{H}(I')|>|I'|$. Again the fact that $\widetilde{H}_K(\varphi_K^+)>\varphi_K^+$ gives a contradiction.
\end{proof}

For $K>K_\theta$ we get three fixed points.

\begin{lemma}\label{s4l9}
Let $\theta \in (0,\pi/2)$.
For $K>K_\theta$ there exist three fixed points of $\widetilde{H}$. There are fixed points $\phi_0, \phi_1$ and $\phi_2$ such that $\phi_2<\phi_0<\phi_1$, $\phi_1$ and $\phi_2$ are repelling and $\phi_0$ is attracting. Further we have $\phi_1,\phi_0\in\widetilde{\mathbb{F}}^+_\theta$ and $\phi_2\in\widetilde{\mathbb{F}}^-_\theta$.
\end{lemma}

\begin{proof}
From Lemma \ref{s4l5} we know there must be exactly one repelling fixed point $\phi_2 \in \widetilde{\mathbb{F}}^-_\theta$. We are left to show there are two fixed points in $\widetilde{\mathbb{F}}^+_\theta$ by Lemma~\ref{s4l1}.

By the methods of Lemma \ref{s4l7}, we see that for
$K>K_\theta$ we have $\widetilde{H}(\phi_{K_\theta})<\phi_{K_\theta}$. Since $\widetilde{H}(\theta) = 2\theta > \theta$, by continuity there exists
some $\phi_0\in(\theta,\phi_{K_\theta})$ which is fixed by $\widetilde{H} = \widetilde{H}_K$.
Similarly as $\widetilde{H}(\pi/2+\theta)=\pi+2\theta>\pi/2+\theta$, there exists $\phi_1\in(\phi_{K_\theta},\pi/2+\theta)$ that is fixed by $\widetilde{H}$. Hence for $K>K_\theta$ we have three fixed points. Note that we can have at most three fixed points since the fixed points of $\widetilde{H}$ correspond to roots of the cubic $P$ given in \eqref{cubic}.

Finally we have that $\phi_0\in J$ and $\phi_1,\phi_2\notin J$ by construction, and so $\phi_0$ is attracting and $\phi_1,\phi_2$ are repelling.
\end{proof}

\begin{figure}[h]
\begin{center}
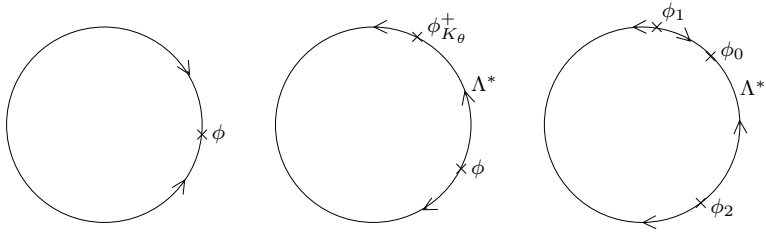
\caption{Diagram showing the local dynamics of one, two and three fixed points.}\label{s4f6}
\end{center}
\end{figure}

The preceding lemmas prove Theorem~\ref{s2t1}.

\section{Pre-images of fixed rays and basins of attraction}

In this section we will prove Theorem \ref{s2t2} by studying the backward orbits of the fixed points of $\widetilde{H}$ and any basins of attraction. We will see that $\widetilde{H}$ restricted to $S^1$ is actually a Blaschke product. We take advantage of this fact, and use properties of Julia sets and Fatou sets of rational functions. In view of Lemma \ref{Upsilon}, throughout this section we assume that $\theta \in [0,\pi/2]$.

\subsection{Basin of attraction}
In the case where $\widetilde{H}$ has two or three fixed points on $S^1$, we will see that the non-repelling fixed point has a basin of attraction. When $\widetilde{H}$ has three fixed points, the basin is formed by a union of open intervals, whereas when $\widetilde{H}$ has two fixed points, the basin is formed by a union of half-open intervals.

\begin{definition}\label{Basin}
The \emph{basin of attraction} $\widetilde{\Lambda}$ of a non-repelling fixed point $\phi$ of $\widetilde{H}$ is given by
\[\widetilde{\Lambda}:=\{\varphi\in S^1\;|\;\widetilde{H}^n(\varphi)\to\phi \mbox{ as } n\to\infty\}.\]
The \emph{immediate basin of attraction} $\widetilde{\Lambda}^*$ is the component of $\widetilde{\Lambda}$ containing $\phi$.
\end{definition}

Recall that we use a tilde to denote sets in $S^1$ and that the basin of attraction of the non-repelling fixed ray $R_{\phi}$ of $H$ in $\C$ will be $\Lambda=\{R_\varphi\;|\;\varphi\in\widetilde{\Lambda}\}$.

\begin{lemma}\label{s4l10}
Recalling the notation of Lemma \ref{s4l001},
suppose $\widetilde{H}$ has two fixed points. Then the neutral fixed point $\phi_{K_{\theta}}$ has an immediate basin of attraction $\widetilde{\Lambda}^{*}$ that is the interval bounded by $\phi_{K_{\theta}}$ and the repelling fixed point $\phi \in \widetilde{\mathbb{F}}_{\theta}^-$.
\end{lemma}

\begin{proof}
From Lemma \ref{s4l7} we know that intervals of the form $[\phi_{K_\theta},\varphi]$ are expanded and so
\[\widetilde{\Lambda}^{*}=(\psi,\phi_{K_\theta}],\]
for some $\psi$. Since the repelling fixed point $\phi\notin\widetilde{\Lambda}^*$, we have $\psi\geq\phi$. However, from Lemma~\ref{s4l7} we also know that all intervals $[\varphi,\phi_{K_{\theta}}]$ such that
\begin{equation}\label{eqbasin}
\phi<\varphi<\phi_{K_{\theta}}
\end{equation}
are contracted under $\widetilde{H}$. Hence if $\varphi$ satisfies \eqref{eqbasin} then $\varphi\in\widetilde{\Lambda}^*$. Therefore we have
\[\widetilde{\Lambda}^{*}=(\phi,\phi_{K_\theta}].\]
\end{proof}

\begin{lemma}\label{s4l11}
When $\widetilde{H}$ has three fixed points $\phi_2<\phi_0<\phi_1$ as in Lemma~\ref{s4l9}, the attracting fixed point $\phi_0$ has an immediate basin of attraction
\[\widetilde{\Lambda}^{*}=(\phi_2,\phi_1).\]
\end{lemma}

\begin{proof}
As $\phi_0$ is an attracting fixed point
\[\widetilde{\Lambda}^{*}=(\varphi_2,\varphi_1),\]
for some $\varphi_2<\phi_0<\varphi_1$. By Lemma~\ref{s4l9} we know that all intervals of the form $[\phi_0,\varphi_1]$ and $[\varphi_2,\phi_0]$, where
\[\phi_2<\varphi_2<\phi_0<\varphi_1<\phi_1,\]
are contracted under $\widetilde{H}$ and so $\varphi_1,\varphi_2\in\widetilde{\Lambda}^*$. Further as $\phi_1,\phi_2\notin\widetilde{\Lambda}^*$ this implies that
\[\widetilde{\Lambda}^{*}=(\phi_2,\phi_1).\]
\end{proof}

\subsection{Writing $\widetilde{H}$ as a Blaschke product}

We will consider the case $\theta=0$, since $h_{K,\theta}$ can be obtained from $h_{K,0}$ by pre-composing and post-composing by the corresponding rotations. Let $h=h_{K,0}$. The induced map $\widetilde{h}$ on $S^1$ cannot be a M\"{o}bius map, however it is $\pi$ periodic and so we can renormalize $\widetilde{h}:(-\pi/2,\pi/2]\to(-\pi/2,\pi/2]$ to a map $\hat{h}:(-\pi,\pi]\to(-\pi,\pi]$ by defining
\begin{equation}\label{s5eh1}
\hat{h}(\varphi)=2\widetilde{h}(\varphi/2).
\end{equation}
This map $\hat{h}$ has an attracting fixed point at $\varphi=0$ and a repelling fixed point at $\varphi=\pi.$

\begin{lemma}\label{s5lm1}
The map $\hat{h}$ agrees with the M\"obius map
\[A_{K}(z) = \frac{z+\alpha}{1+\overline{\alpha} z},\]
on $S^1$, where $\alpha = (K-1)/(K+1)$.
\end{lemma}

\begin{proof}
Recall that the induced map is given by $\widetilde{h}(\varphi)=\tan^{-1}(\tan(\varphi)/K)$ and so
\begin{align*}
\tan(\hat{h}(\varphi))&=\tan(2\tan^{-1}(\tan(\varphi/2)/K))\\
&=\frac{2\tan(\varphi/2)/K}{1-\tan^2(\varphi/2)/K^2}\\
&=\frac{(2K\sin\varphi)/(1+\cos\varphi)}{K^2-(\sin^2\varphi)/(1+\cos(\varphi))^2}\\
&=\frac{2K\sin\varphi}{(K^2-1)+(K^2+1)\cos\varphi}.
\end{align*}

Define $A_{K}$ to be the M\"{o}bius map
\begin{equation}\label{s5e00}
A_{K}(z)=\frac{z+\alpha}{1+\overline{\alpha}z},
\end{equation}
where $\alpha=(K-1)/(K+1)$. By construction $A_{K}(S^1)=S^1$. Further by writing $z=x+iy$ we see
\begin{align*}
A_{K}(x+iy)&=\frac{(x+\alpha+iy)(1+\alpha x-i\alpha y)}{(1+\alpha x+i\alpha y)(1+\alpha x-i\alpha y)}\\
&=\frac{x(1+\alpha^2)+\alpha(1+x^2+y^2)+i y(1-\alpha^2)}{(1+\alpha x+i\alpha y)(1+\alpha x-i\alpha y)}.
\end{align*}
Noticing that $x,y\in S^1$ and so $x^2+y^2=1$ and the fact that
\[\arg(A_{K}(z))=\tan^{-1}[\Im(A_{K,0}(z))/\Re(A_{K,0}(z))],\]
we have
\[\tan[\arg(A_{K}(z))]=\frac{y(1-\alpha^2)}{2\alpha+x(1+\alpha^2)}.\]
It is easy to see that if $\varphi$ denotes the argument of the point $z\in S^1$ and $z=x+iy$, then $x=\cos\varphi$ and $y=\sin\varphi$. Hence using $A_{K}$ to denote the map $A_{K}$ induces on the argument of $z$ we see
\[A_{K}(\varphi)=\frac{2K\sin\varphi}{(K^2-1)+\cos\varphi(K^2+1)}=\hat{h}(\varphi).\]
This shows that $\hat{h}$ is a M\"{o}bius map of $S^1$.
\end{proof}

\begin{figure}[h]
\begin{center}
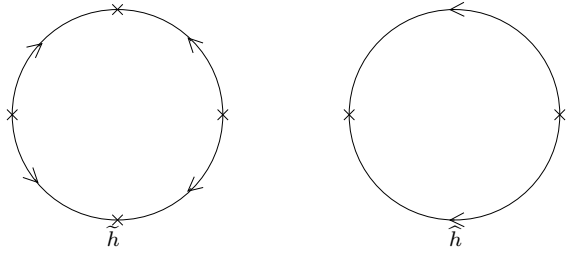
\caption{A diagram for $\theta=0$ showing how we obtain a M\"{o}bius map with two fixed points.}\label{s5f1}
\end{center}
\end{figure}

\begin{lemma}\label{s5l1}
Let $H=H_{K,\theta}$. Then $\widetilde{H}:S^1\to S^1$ agrees with the Blaschke product $B$ on $S^1$ given by
\[ B(z)=\frac{z^2+\mu}{1+\overline{\mu}z^2} =
\left ( \frac{z-a}{1-\overline{a}z} \right )\left( \frac{z+a}{1+\overline{a}z} \right)
,\]
where $\mu=e^{2i\theta} \left ( \frac{K-1}{K+1} \right )$ is the complex dilatation of $H$ and $a = e^{i(\theta - \pi/2)} \left ( \frac{K-1}{K+1} \right )^{1/2}$.
\end{lemma}

\begin{proof}
For $\varphi \in (-\pi/2 +\theta, \pi/2, +\theta ]$, we have $\widetilde{h}_{K,\theta}(\varphi) = \widetilde{h}_{K,0}(\varphi - \theta)+\theta$.
Using Lemma \ref{s5lm1}, we see that
\begin{align*}
\widetilde{h}_{K,\theta}(\varphi) &=\frac{ \hat{h}(2\varphi - 2\theta)}{2} + \theta \\
&= \frac{ A_{K}(2\varphi - 2\theta)}{2} + \theta
\end{align*}
As $\widetilde{H}_{K,\theta}=2\widetilde{h}_{K,\theta}$ we obtain
\[ \widetilde{H}_{K,\theta}(\varphi)=2\theta + A_{K}(2\varphi-2\theta)\]
for $\varphi\in(-\pi/2+\theta,\pi/2+\theta]$, and by $\pi$-periodicity, for the remaining values of $\varphi$, we have
\begin{align*}
\widetilde{H}_{K,\theta}(\varphi)&=2\theta + 2\pi + A_{K}(2\varphi+2\pi-2\theta) \\
&= 2\theta + A_{K}(2\varphi-2\theta).
\end{align*}
Letting $z=e^{i\varphi}$, recalling \eqref{s5e00} we see
\begin{align}
\widetilde{H}_{K,\theta}(e^{i\varphi})&=e^{2i\theta}\frac{e^{i(2\varphi-2\theta)}+\alpha}{1+\alpha e^{i(2\varphi-2\theta)}}\nonumber\\
&=\frac{e^{2i\varphi}+\alpha e^{2i\theta}}{1+ \alpha e^{-2i\theta}e^{2i\varphi}}\nonumber\\
&=\frac{z^2+\mu}{1+\overline{\mu}z^2},\label{s5ez}
\end{align}
where $\mu=e^{2i\theta}(K-1)/(K+1)$ is the complex dilatation of $H$.

\end{proof}

\subsection{Proof of Theorem \ref{s2t2}.}

We can now use the standard results on the iteration theory of Blaschke products as given in Proposition~\ref{Blaschke}. Our Blaschke product $B$ has three fixed points, counting multiplicity. Write $J(B)$ and $F(B)$ for the Julia and Fatou sets of $B$, respectively. Note that Theorem~\ref{s2t1} tells us how many fixed points $B$ has on $S^1$.

Suppose that $\widetilde{H}$ has one fixed point $\phi\in S^1$.
We know from Lemma~\ref{s5l1} and Proposition~\ref{Blaschke} that the Julia set of $B$ is $S^1$ and so $\phi\in J$.
Since $J(B) = \overline{O^-(z)}$ for any $z \in J(B)$,
we immediately have that $\{\widetilde{H}^{-k}(\phi)\}_{k=0}^{\infty}$ is dense in $S^1$.

Suppose that $\widetilde{H}$ has more than one fixed point. Let $\phi$ be the non-repelling fixed point.
By Lemma \ref{Blaschke} we know that the Julia set $J(B)$ of $B$ is a Cantor subset of $S^1$. This implies that $E=F(B)\cap S^1$ is a dense subset of $S^1$.
Consider a point $z\in E$, then any neighbourhood $U\subset F(B)$ of $z$ contains points in $U\cap\D$. By the Denjoy-Wolff Theorem,
we have $B^n(w)\to\phi$ as $n\to\infty$ for every $w \in U \cap \D$. As $U$ is contained in the Fatou set, the iterates $\{ B^n\}$ are a  normal family on $U$ and so $B^n(z)\to\phi$ as $n\to\infty$. This implies $z\in\widetilde{\Lambda}$, and hence $\widetilde{\Lambda}$ is dense in $S^1$.

This completes the proof of Theorem \ref{s2t2}.

\section{Proof of Corollary \ref{s2c1}}

Fix $K>1$ and $\theta \in (-\pi/2, \pi/2]$. By \cite[Theorem 4.3]{FG}, the escaping set $I(H)$ is a connected, completely invariant, open neighbourhood of infinity and $\partial I(H)$ is a completely invariant closed set. The point $0$ is clearly fixed by $H$ and since
\[ |z|^2 \leq |H(z)| \leq K^2|z|^2,\]
there is a neighbourhood of $0$ contained in the basin of attraction $\mathcal{A}(0)$. It is therefore clear that $\mathcal{A}(0)$ is completely invariant and open.

Let $R_{\phi}$ be a fixed ray of $H$. Then on $R_{\phi}$, we have
\[ H(re^{i\phi}) = \alpha r^2 e^{i\phi},\]
where $\alpha = (1+(K^2-1) \cos^2(\phi-\theta))$ by the polar form \eqref{hpolar} of $H$. For $r=1/\alpha$, this point is fixed, for $r>1/\alpha$ the point is in $I(H)$ and for $r<1/\alpha$, the point is in $\mathcal{A}(0)$. By complete invariance, any pre-image of $R_{\phi}$ breaks up into $\mathcal{A}(0), I(H)$ and $\partial I(H)$ in the same way.

Assume that $K<K_{\theta}$, then by Theorem \ref{s2t1} $H$ has one fixed ray and by Theorem \ref{s2t2} the set $\{ H^{-k}(R_{\phi}):k \geq 0 \}$ is dense in $\C$. Since $\mathcal{A}(0)$ is open, this proves the result in this case.

On the other hand, if $K\geq K_{\theta}$, then by Theorem \ref{s2t1} write $\Lambda$ for the basin of attraction of the non-repelling fixed ray $R_{\phi}$.
By Theorem \ref{s2t2}, $\Lambda$ is dense in $\C$. Suppose that $R_{\varphi} \in \Lambda$. Then $H^n(R_{\varphi}) \to R_{\phi}$. Since $\mathcal{A}(0)$ and $I(H)$ are open, it is not hard to see that $R_{\varphi}$ breaks up in the same way that $R_{\phi}$ does. Since $\Lambda$ is dense in $\C$, the openness of $\mathcal{A}(0)$ again implies the result in this case, which proves the corollary.

\section{Nowhere Uniformly Quasiregular Mappings}

We now prove Theorem \ref{s2t3}, that $H$ is nowhere uniformly quasiregular. This will be an application of Theorem~\ref{s2t2}. In \cite{FF} we showed $H$ was not uniformly quasiregular by considering points on a fixed ray. Here, we will use density of pre-images of the fixed ray in the one fixed ray case, and density of the basin of attraction in the remaining cases. Let us first define what we mean by nowhere uniformly quasiregular.

\begin{definition}\label{s6d1}
We define the distortion of a function $f:\C\to\C$ at a point $z\in\C$ as:
\[K_z(f):=\limsup_{\mathrm{diam}(U)\to 0} K(f|U),\]
where $U$ is any neighbourhood of $z$ and the $\limsup$ is taken as the diameters of these neighbourhoods tend to 0.
\end{definition}

\begin{definition}\label{s6d2}
A function $f:\C\to\C$ is \emph{nowhere uniformly quasiregular} if
\[K_z(f^n)\to\infty \mbox{ as } n\to\infty \mbox{ for all } z\in\C.\]
\end{definition}

Note the difference between $K_z(f)$ and $K_f(z)$.

\begin{lemma}\label{s6l2}
If $H$ has one fixed ray $R_\phi$ then $H$ is nowhere uniformly quasiregular.
\end{lemma}

\begin{proof}
Fix $z\in\C$. Theorem \ref{s2t2} tells us that $\mathcal{P} = \{ H^{-k}(R_{\phi}) \}$ is dense in $\C$. If $z$ lies on a ray $R_\varphi\in\mathcal{P}$ then there must exist some $m$ such that $H^m(R_\varphi)=R_\phi$. That is, $H^m(z)$ lies on the ray $R_\phi$. We can apply the formula for the complex dilatation of the composition of functions, from for example \cite{FM}, to obtain:
\begin{equation}\label{s6e1}
\mu_{H^n\circ H^m}(z)=\frac{\mu_{H^m}(z) + r_{H^m}(z) \mu_{H^n}(H^m(z))}{1+r_{H^m}(z)\overline{\mu_{H^m}(z)}\mu_{H^n}(H^m(z))},
\end{equation}
where $r_{H^m}(z)=\overline{(H^m)_{z}(z)}/(H^m)_z(z)$. Notice that $|r_{H^m}(z)|=1$ and that if we define
\begin{equation}\label{s6e2}
B(w):=r_{H^m}(z)\left(\frac{w + \mu_{H^m}(z)\overline{r_{H^m}(z)}}{1+\overline{[\overline{r_{H^m}(z)}\mu_{H^m}(z)]}w}\right),
\end{equation}
then $B$ is a M\"{o}bius map of the disk and further we see that
\[B[\mu_{H^n}(H^m(z))]=\mu_{H^n\circ H^m}(z),\]
for $n \geq 1$.
Using the fact that $H^{n+m}(z) \in R_{\phi}$ for $n \geq 0$, \eqref{s6e2}, Proposition~\ref{s6l1} and the remark afterwards we see that \eqref{s6e1} becomes
\begin{equation}\label{s6e3}
\mu_{H^n\circ H^m}(z)=B(A^{n-1}(\mu_{H^m}(z))),
\end{equation}
for $n \geq 1$.
We know $|A^n(w)|\to 1$ as $n\to\infty$ for any $w \in \D$, $B(\partial \D) = \partial \D$ and so we have
\[|\mu_{H^\ell(z)}(z)|\to 1 \mbox{ as } \ell\to\infty.\]
Any neighbourhood $U\ni z$ trivially contains $z$ and so $K_z(H^\ell)$ is unbounded as $\ell \to \infty$.

Next suppose $z$ lies on a ray not in $\mathcal{P}$. As $\mathcal{P}$ is dense in $\C$, any neighbourhood $U\ni z$ must intersect a ray $R_\varphi\in\mathcal{P}$. Picking one such ray there must exist $m$ (depending on the neighbourhood $U$) such that $H^m(R_\varphi)=R_\phi$ and we can apply the same argument above to conclude $K_z(H^\ell)$ is unbounded as $\ell \to \infty$ for any $z\in\C.$
\end{proof}

When we have more than one fixed ray we don't have that any neighbourhood of a point contains the pre-image of a fixed ray but we do know that if we iterate $H^n(z)$ then we either end up on a fixed ray, or the argument of $H^n(z)$ tends to the argument of the non-repelling fixed ray. We take advantage of this to prove the remainder of Theorem~\ref{s2t3}.

\begin{lemma}\label{s6l3}
If $H$ has more than one fixed ray then $H$ is nowhere uniformly quasiregular.
\end{lemma}

\begin{proof}
Fix $z\in\C$. From Theorem \ref{s2t2} we know that either $z$ lies on the preimage of a fixed ray, or $z\in\Lambda.$ In the first case the result follows from the methods of the previous lemma. In the second case we know that the argument of $H^n(z)$ tends to the argument of the non-repelling fixed ray $\phi$ as $n\to\infty$.

We define the sequence $\phi_n\in S^1$ by $H^n(z)\in R_{\phi_n}$. Then $\phi_n \to \phi$ as $n\to\infty$, where $\phi$ is the non-repelling fixed point of $\widetilde{H}$.
Again we use the formula for the complex dilatation of composition of functions from \cite{FM} reformulated slightly differently than in  \eqref{s6e1} to see
\[\mu_{H^n}(z)=\mu_{H^{n-1}\circ H}(z) =\frac{\mu_{H}(z)+r_{1}(z)\mu_{H^{n-1}}(H(z))}{1+r_{1}(z)\overline{\mu_{H}(z)}\mu_{H^{n-1}}(H(z))}.\]
Recalling that $\mu_H$ is constant, we can write
\[ \mu_{H^n}(z) = A_1(\mu_{H^{n-1}}(H(z))),\]
where $A_1$ is the M\"{o}bius map
\[ A_1(w) = \frac{ \mu_H + r_H(z)w}{1+r_H(z)\overline{\mu_H}w}.\]
Using the same method, we may write
\[\mu_{H^{n-1}}(H(z)) = A_2(\mu_{H^{n-2}}(H^2(z)),\]
where $A_2$ is the M\"{o}bius map
\[ A_2(w) = \frac{ \mu_H + r_H(H(z))w}{1+r_H(H(z))\overline{\mu_H}w}.\]
By induction, we may write
\[ \mu_{H^n}(z) = A_1\circ A_2\circ\ldots\circ A_{n-1}(\mu_H(H^{n-1}(z))),\]
where each $A_i$ is a M\"{o}bius map given by
\[ A_i(w)=\frac{ \mu_H + r_H(H^{i-1}(z))w}{1+r_H(H^{i-1}(z))\overline{\mu_H}w}.\]
It is not hard to see that $H_z(z) = (K+1)h(z)$, and so
\[ r_H(H^{j-1}(z)) = \exp ( -2i \arg[h(H^{j-1}(z))] ).\]
As $j \to \infty$, we have $\arg [h(H^{j-1}(z))] \to \arg [h(re^{i\phi})]$ for any $r>0$.
Since $\phi$ is a fixed ray of $H$, $\arg [h(re^{i\phi})] = \phi /2$.
In particular, we have that the M\"{o}bius maps $A_i$ converge to the M\"{o}bius map
\[ A(z) = \frac{ \mu_H + e^{-i\phi}w}{1+e^{-i\phi}\overline{\mu_H}w}.\]
By Proposition \ref{s6l1}, $A$ is a hyperbolic M\"{o}bius map with fixed point $\alpha \in \partial \D$.
Recalling $\mu_H(z)=e^{2i\theta}(K-1)/(K+1)$ for all $z\in\C$, we can write
\begin{equation}\label{s6e6}
\mu_{H^n}(z)=A_1\circ A_2\circ\ldots\circ A_{n-1}(\mu_H)=:t_{n-1}(\mu_H).
\end{equation}
Then by Theorem~\ref{sMlMM}, $\mu_{H^n}(z) \to \alpha$ and in particular $|\mu_{H^n}(z)| \to 1$. This proves the lemma.
\end{proof}

Note that we fixed $z$ at the beginning of the proof and that a different choice of $z$ will give rise to different M\"{o}bius maps $A_i$. Together Lemmas~\ref{s6l2}~and~\ref{s6l3} prove Theorem~\ref{s2t3}.

\section{Failure of quasiconformal equivalence}

\subsection{Outline}
Our aim in this section is to prove Theorem \ref{s2t4} and show that for most pairs of values of $K_i,\theta_i$, the corresponding maps $H_i$ are not quasiconformally equivalent on any neighbourhood of infinity, and so neither are $H_i+c_i$ for any choice of $c_i \in \C$ via the B\"{o}ttcher coordinate.

The outline of our strategy is as follows.
\begin{itemize}
\item Each fixed ray $R_{\phi}$ of $H$ has a corresponding hyperbolic M\"{o}bius automorphism of $\D$ which encodes how the complex dilatation of the iterates $H^n$ behaves on $R_{\phi}$.
\item If there is a quasiconformal equivalence $\Psi$ between $H_1$ and $H_2$ such that $\Psi \circ H_1 = H_2 \circ \Psi$ on a neighbourhood of infinity $U$, then $1/C\leq K_{H_1^n}(z) / K_{H_2^n}(\Psi (z)) \leq C$ for some constant $C>0$, all $n\in\N$ and all $z \in U$.
\item We show that in the various cases of different numbers of fixed rays, if there is a quasiconformal equivalence $\Psi$, then the image of a fixed ray of $H_1$ under $\Psi$ will either be a fixed ray of $H_2$, intersect a fixed ray of $H_2$ or converge to a fixed ray of $H_2$.
\item In each case, by comparing the behaviour of the corresponding M\"{o}bius maps for the respective fixed rays, we show that if the corresponding traces are different, then there can be no quasiconformal equivalence.
\end{itemize}

\subsection{Consequences of a quasiconformal equivalence}

Through this section, we will consider two maps $H_1,H_2$ associated with $K_i,\theta_i$ for $i=1,2$.

Recall that
two maps $f_1, f_2 :\C \to \C$ are quasiconformally equivalent on a neighbourhood $U$ of infinity if there exists a neighbourhood of infinity $V$ and a quasiconformal map $\Psi:U \to V$ such that
\[\Psi^{-1} \circ f_2 \circ \Psi(z) = f_1(z),\]
for all $z \in U$.

From Theorem~\ref{thmBot} we know that $H(z)$ and $H(z)+c$ are quasiconformally equivalent on a neighbourhood of infinity. Therefore, if we are interested in knowing when $H_i + c_i$ can be quasiconformally equivalent for $i=1,2$, we can reduce to the situation where $c_i=0$.

If $H_1,H_2$ are quasiconformally equivalent on a neighbourhood of infinity, then
\begin{equation}
(H_2)^{n}(\Psi(z)) = \Psi((H_1)^{n}(z)), \label{s7e3}
\end{equation}
for all $n \in \N$ and $z \in U$.

\begin{lemma}\label{QCE}
If $H_1$ and $H_2$ are quasiconformally equivalent on a neighbourhood $U$ of infinity, then there exists $C>0$ such that
\begin{equation}
\frac{1}{C}\leq\frac{K_{H_1^n}(z)}{K_{H_2^n}(\Psi(z))}\leq C,\label{s7e4}
\end{equation}
for all $n \in \N$ and $z \in U$.
\end{lemma}

\begin{proof}
This follows immediately from \eqref{s7e3} and the fact that distortion is sub-multiplicative with respect to composition, see for example \cite{FM}. We may take $C = (K_{\Psi})^2$.
\end{proof}

Recall from \cite{FF} that the complex dilatation $\mu_n$ of $H^n$ on a fixed ray $R_\phi$ is given by $\mu_n=A^{n-1}(\mu_H)$. Here $\mu_H$ is the complex dilatation of $H$, given by $\mu_H = e^{2i\theta}(K-1)/(K+1),$ and $A$ is the M\"{o}bius automorphism of $\D$ given by
\[A(z):=\frac{\mu_H+e^{-i\phi}z}{1+e^{-i\phi}\overline{\mu_H}z}.\]
It was shown in \cite{FF} that $A$ is a hyperbolic M\"{o}bius map, which means $\tr(A)^2 > 4$. The squared trace of $A$ is
\begin{equation}
\label{tracesquared}
\tr(A)^2 = \frac{(K+1)^2}{2K} (1+ \cos \phi).
\end{equation}

We will use $A_i$ to denote M\"{o}bius map corresponding to the fixed ray $R_{\phi_i}$ of $H_1$ and $B_i$ for the M\"{o}bius map corresponding to the fixed ray $R_{\psi_j}$ of $H_2$.

\begin{lemma}\label{s7ltend}
Let $H_1$ and $H_2$ be quasiconformally equivalent on a neighbourhood $U$ of infinity. Then if $R_{\phi_i}$ and $R_{\psi_j}$, are fixed rays of $H_1$ and $H_2$ respectively and
\[\arg[H_2^n(\Psi(z))] \rightarrow \psi_j\]
for some $z \in R_{\phi_i}$, then
\[\tr(A_i)^2=\tr(B_j)^2.\]
\end{lemma}

Note that this lemma takes care of the cases where $\Psi(R_{\phi_i})$ is either a fixed ray, intersects a fixed ray in one point (which means it must intersect in infinitely many) or is a curve which converges to a fixed ray of $H_2$. This lemma is our key tool in this section.

\begin{proof}
Suppose for a contradiction that $\arg[H_2^n(\Psi(z))] \rightarrow \psi_j$ for some $z \in R_{\phi_i}$, but $\tr(A_i)^2\neq\tr(B_j)^2$.
By Lemma~\ref{sMl3}
\begin{equation}\label{s7ne1}
d_h(0,A_i^n(\mu_{H_1}))=\log\left[O(1/k^n)\right],
\end{equation}
where
\[ k=(\tr(A_i)^2-2-(\tr(A_i)^4-4\tr(A_i)^2)^{\frac{1}{2}})/2.\]
Then by Proposition \ref{s6l1} we see that
\begin{align*}
K_{H_1^n}(z) &= \frac{1+|\mu_{H_1^n}(z)|}{1-|\mu_{H_1^n}(z)|}\\ &=\exp \left ( d_h(0,A_i^{n-1}(\mu_{H_1})) \right ) \\ &=O(e^{\log[(1/k^{n-1})]})\\ &=O(1/k^{n-1}).
\end{align*}

By hypothesis, $H_2^n(\Psi(z)) \in R_{\gamma_n}$ for some sequence of rays $R_{\gamma_n}$ where $\gamma_n \to \psi_j$.
As in Lemma \ref{s6l3}, we may write
\[ \mu_{H_2^n}(\Psi (z)) = B_1 \circ \ldots \circ B_{n-1} [ \mu_{H_2} ( H_2^{n-1} (\Psi(z) ) ) ],\]
where $\mu_{H_2}$ is a constant and each $B_m$ is a M\"{o}bius map given by
\[ B_m(w) = \frac{ \mu_{H_2} + r_{H_2}( H_2^{m-1}(\Psi(z )) )w }{ 1+
r_{H_2}( H_2^{m-1}(\Psi(z )) ) \overline{\mu_{H_2}} w},\]
and we have $r_{H_2}( H_2^{m-1}(\Psi(z )) ) \to e^{-i\psi_j}$. Hence $B_m \to B_j$ as $m \to \infty$.
Let $t_n = B_1 \circ \ldots \circ B_{n-1}$.
Then by Lemma~\ref{sMl3}
\[d_h(0,t_n( \mu_{H_2}(\Psi(z))))=\log\left[\frac{1}{\prod_{j=1}^{n-1}\ell^j}\right] +O(1),\]
where $\ell_m\to \ell$ as $m\to\infty$. Here, $\ell$ is the quantity
from Lemma \ref{sMl1} involving the trace squared of the M\"{o}bius map $B_j$ corresponding the fixed ray $R_{\psi_j}$. By our hypothesis, $k \neq \ell$.
By Lemma~\ref{sMl3}, we have
\begin{align}
K_{H_2^n}(\Psi(z))&=\frac{1+|\mu_{H_2^n}(\Psi(z))|}{1-|\mu_{H_2^n}(\Psi(z))|}\notag\\
&=\exp \left ( d_h(0,t_{n-1}(\mu_{H_2}(\Psi(z)))) \right )\notag\\
&=O \left ( e^{\log\left[\left(\frac{1}{\prod_{j=1}^{n-1}\ell_j}\right)\right]} \right )\notag\\
&=O\left(\frac{1}{\prod_{j=1}^{n-1}\ell_j}\right).\label{s7ne2}
\end{align}
As $\ell_j\to \ell\neq k$ then for all $\varepsilon>0$ there exists $N\in\N$ such that if $k<\ell$ then $\ell_j/k\geq\alpha>1$ for all $j>N$ and if $k>\ell$ then $\ell_j/k\leq\beta<1$ for all $j>N$. So first if $k>\ell$
\[\frac{1}{k^n}\Bigg{/}\frac{1}{\prod_{i=1}^n \ell_j}\leq\left(\frac{\prod_{i=1}^N \ell_j}{k^N}\right)\beta^{n-N}\to 0 \mbox{ as } n\to \infty \]
and if $k<\ell$ then,
\[\frac{1}{k^n}\Bigg{/}\frac{1}{\prod_{i=1}^n \ell_j}\geq\left(\frac{\prod_{i=1}^N \ell_j}{k^N}\right)\alpha^{n-N}\to \infty \mbox{ as } n\to\infty.\]
In either case, we contradict Lemma~\ref{QCE}.
\end{proof}

\subsection{The case where $H_2$ has one fixed ray}

We next show that if one of our mappings has one fixed ray, then a quasiconformal equivalence implies the other mapping must have one fixed ray.

\begin{lemma}\label{s7l3}
Suppose $H_2$ has one fixed ray, and $H_1$ has more than one fixed ray. Then $H_1$ and $H_2$ are not quasiconformally equivalent.
\end{lemma}

\begin{proof}
Suppose $H_2$ has one fixed ray $R_\psi$, $H_1$ has two or three fixed rays $R_{\phi_0},R_{\phi_1}$ and possibly $R_{\phi_2}$ and there is a quasiconformal equivalence $\Psi$ between them. If $\Psi(R_{\phi_i})$ is a ray then \eqref{s7e3} implies that it must be fixed by $H_2$, but as there is only one fixed ray $R_\psi$ of $H_2$ this implies $\Psi(R_{\phi_i})=R_\psi$ and so $\Psi(R_{\phi_j})$ cannot be a ray for $j\neq i$. Since the image of this, $\Psi(R_{\phi_j})$, is not a ray then by Theorem~\ref{s2t2} it must intersect $\{ H_2^{-k}(R_{\psi}) \}$ and so by \eqref{s7e3} must intersect $R_\psi$ contradicting $\Psi$ being injective.

Therefore $\Psi(R_{\phi_i})$ is not a ray for any $i$ and hence again by Theorem~\ref{s2t2}, there exists $z_i\in R_{\phi_i}$ such that $\Psi(z_i)\in R_\psi$. We can apply Lemma~\ref{s7ltend} to see that the corresponding traces squared of the M\"{o}bius maps $A_i$ of the fixed rays $R_{\phi_i}$ must equal the trace squared of the fixed ray $R_{\psi}$. Therefore $\tr(A_i)^2=\tr(A_j)^2$ for each $i,j$. Recall from \eqref{tracesquared} that
\begin{equation*}
\tr(A_i)^2=\frac{(K_1+1)^2(1+\cos\phi_i)}{2K_1}.
\end{equation*}
As $K_1$ is fixed, this just depends on $\cos\phi$. To finish the proof, we need to show that $\cos\phi_i\neq\cos\phi_j$ for some pair of fixed rays of $H_1$, to give us a contradiction. Equivalently, we need to show that $\phi_i \neq -\phi_j$ for some pair of fixed rays of $H_1$. We will do this in Lemmas~\ref{s4lT0}, \ref{s4lT1} and \ref{s4lT2} below.
\end{proof}

We have already seen in Lemma~\ref{theta0} that when $\theta=0$ the repelling fixed points satisfy $\phi_1=-\phi_2$ and the attracting fixed point is always $\phi_0=0$. We now consider the remaining cases when $\theta\neq0$. The following result tells us that if a point of $S^1$ is moved a given amount by the map $\widetilde{h}$ induced by $h$, then there are only two possibilities.

\begin{lemma}\label{s4lT0}
Define $G$ by
$G(\varphi):=\varphi-\widetilde{h}(\varphi)$ for $\varphi-\theta\in(0,\pi/2)$. Then $G(\varphi_1)=G(\varphi_2)$ implies that either
\[\varphi_1=\varphi_2 \mbox{ or } \varphi_1=\tan^{-1}\left(K/\tan(\varphi_2-\theta)\right)+\theta.\]
\end{lemma}

\begin{proof}
Suppose $G(\varphi_1)=G(\varphi_2)$. Then this implies
\[\varphi_1-\tan^{-1}\left(\frac{\tan(\varphi_1-\theta)}{K}\right)+\theta=\varphi_2-\tan^{-1}\left(\frac{\tan(\varphi_2-\theta)}{K}\right)+\theta,\]
which we rearrange and use the addition formula for $\tan^{-1}$ to yield
\begin{equation}\label{s4eT1}
\varphi_1-\varphi_2=\tan^{-1}\left(\frac{K(\tan(\varphi_1-\theta)-\tan(\varphi_2-\theta))}{K^2+\tan(\varphi_1-\theta)\tan(\varphi_2-\theta)}\right).
\end{equation}
We can add and subtract $\theta$ to the left hand side then apply $\tan$ to both sides and use the addition formula for $\tan$ and the fact $\tan$ is an odd function to see that \eqref{s4eT1} is equivalent to
\begin{equation}\label{s4eT2}
\frac{\tan(\varphi_1-\theta)-\tan(\varphi_2-\theta)}{1+\tan(\varphi_1-\theta)\tan(\varphi_2-\theta)}= \frac{K(\tan(\varphi_1-\theta)-\tan(\varphi_2-\theta))}{K^2 +\tan(\varphi_1-\theta)\tan(\varphi_2-\theta)}.
\end{equation}
We can then rearrange \eqref{s4eT2} to get
\begin{equation}\label{s4eT3}
K^2[\tan(\varphi_1-\theta)-\tan(\varphi_2-\theta)]+\tan^2(\varphi_1-\theta)\tan(\varphi_2-\theta)-\tan(\varphi_1-\theta)\tan^2(\varphi_2-\theta)
\end{equation}
\[=K[\tan^2(\varphi_1-\theta)\tan(\varphi_2-\theta)-\tan(\varphi_1-\theta)\tan^2(\varphi_2-\theta)+\tan(\varphi_1-\theta)-\tan(\varphi_2-\theta)].\]
Rearranging and factorising \eqref{s4eT3} we obtain
\begin{equation}\label{s4eT4}
(\tan(\varphi_1-\theta)-\tan(\varphi_2-\theta))[K^2-K+\tan(\varphi_1-\theta)\tan(\varphi_2-\theta)(1-K)]=0.
\end{equation}
This shows that $\tan(\varphi_1-\theta)=\tan(\varphi_2-\theta)$ is a solution, which for our range of possible values implies $\varphi_1=\varphi_2$. The other solutions are given by
\begin{equation}\label{s4eT5}
\frac{K^2-K}{K-1}=\tan(\varphi_1-\theta)\tan(\varphi_2-\theta).
\end{equation}
We know $K\neq1$ hence we can cancel $K-1$ on the left hand side of \eqref{s4eT5} and rearrange to obtain
\begin{equation}
\label{s4eTend}
\varphi_1=\tan^{-1}(K/\tan(\varphi_2-\theta))+\theta.
\end{equation}
\end{proof}

We show that if $H$ has two fixed rays, then they cannot be symmetric about the real axis.

\begin{lemma}\label{s4lT1}
Let $\theta\neq0$ and $K>1$. If the corresponding map $H$ has two fixed rays $R_{\phi_1}$ and $R_{\phi_2}$ then
\[\phi_1\neq-\phi_2.\]
\end{lemma}

\begin{proof}
Assume to the contrary that $\phi_1>0$ and $\phi_2=-\phi_1$. As $\phi_1$ is a neutral fixed point from Lemma \ref{s4l7}, we have that $\widetilde{H}'(\phi_1)=1$. From \eqref{diff} this implies
\begin{equation}\label{fixeq1}
\phi_1=\cos^{-1}\left[\left(\frac{2K-1}{K^2-1}\right)^{\frac{1}{2}}\right]+\theta.
\end{equation}
Further we know that the $\phi_i$ are fixed and also that they are moved the same magnitude under $\widetilde{h}$. These imply
\begin{equation}\label{fixeq2}
\widetilde{h}(\phi_1)=-\widetilde{h}(\phi_2)
\end{equation}
and
\begin{equation}\label{fixeq3}
G(\phi_1)=-G(\phi_2).
\end{equation}
By reflecting in the $\theta$ axis we see that for $\varphi-\theta\in(-\pi/2,0)$ we have
\[G(\varphi)=-G(-\varphi+\theta).\]
Hence \eqref{fixeq3} implies
\begin{equation}\label{fixeq4}
G(\phi_1)=G(\phi_1+\theta)
\end{equation}
We can apply Lemma~\ref{s4lT0} with $\varphi_1=\phi_1+\theta$ and $\varphi_2=\phi_1$ to see
\begin{equation}\label{fixeq5}
\phi_1+\theta=\tan^{-1}\left(K/\tan(\phi_1-\theta)\right)+\theta.
\end{equation}
Substituting \eqref{fixeq1} into \eqref{fixeq5}, writing $X=\left(\frac{2K-1}{K^2-1}\right)^{\frac{1}{2}}$ and rearranging we see
\begin{equation}\label{fixeq6}
\tan[\cos^{-1}X]\tan[\cos^{-1}(X+\theta)]=K.
\end{equation}
Next apply the addition formula to $\tan[\cos^{-1}(X+\theta)]$ to see \eqref{fixeq6} becomes
\begin{equation}\label{fixeq7}
\tan[\cos^{-1}X]\frac{\tan[\cos^{-1}X]+\tan\theta}{1-\tan[\cos^{-1}X]\tan\theta}=K.
\end{equation}
Let $Y=\tan[\cos^{-1}X]$ and rearrange \eqref{fixeq7} to obtain
\begin{equation}\label{fixeq8}
\tan\theta=\frac{K-Y^2}{Y(K+1)}.
\end{equation}

Next we substitute \eqref{fixeq1} into \eqref{fixeq2}, and again write $X=\left(\frac{2K-1}{K^2-1}\right)^{\frac{1}{2}}$ to see
\begin{equation}\label{fixeq9}
\tan^{-1}\left(\frac{\tan[\cos^{-1}X]}{K}\right)+\theta=-\tan^{-1}\left(\frac{\tan[-\cos^{-1}X-2\theta]}{K}\right)-\theta.
\end{equation}
Rearranging \eqref{fixeq9} and using the fact $\tan$ and $\tan^{-1}$ are odd functions, we obtain
\begin{equation}\label{fixeq10}
2\theta=\tan^{-1}\left(\frac{\tan[\cos^{-1}X+2\theta]}{K}\right)-\tan^{-1}\left(\frac{\tan[\cos^{-1}X]}{K}\right).
\end{equation}
Next we apply the addition formula for $\tan^{-1}$ to \eqref{fixeq10} and then apply $\tan$ to both sides to obtain
\begin{equation}\label{fixeq11}
\tan2\theta=\frac{K(\tan[\cos^{-1}X+2\theta]-\tan[\cos^{-1}X])}{K^2+\tan[\cos^{-1}X+2\theta]\tan[\cos^{-1}X]}.
\end{equation}
Rearranging \eqref{fixeq11}, applying the addition formula to $\tan[\cos^{-1}X+2\theta]$ and writing $Y~=~\tan[\cos^{-1}X]$ we see
\begin{equation}
\tan2\theta\left(K^2 + Y\frac{Y+\tan2\theta}{1-Y\tan2\theta}\right)=K\left(\frac{Y+\tan2\theta}{1-Y\tan2\theta} -Y\right).
\end{equation}
Canceling the denominators $1-Y\tan2\theta$ in \eqref{fixeq11} we see
\begin{equation}\label{fixeq12}
\tan2\theta(K^2(1-Y\tan2\theta)+Y(Y+\tan2\theta)=K(Y+\tan2\theta -Y(1-Y\tan2\theta)).
\end{equation}
Expanding and canceling \eqref{fixeq12} then rearranging we obtain
\begin{equation}\label{fixeq13}
\tan2\theta=\frac{Y^2+K^2-Y^2K-K}{Y(K^2-1)}=\frac{K-Y^2}{Y(K+1)}.
\end{equation}
Together \eqref{fixeq8} and \eqref{fixeq13} imply  $\tan\theta=\tan2\theta$. Letting $\tan\theta=T=(K-Y^2)/(Y(K+1))$ and using the double angle formula we must have
\begin{equation}\label{fixeq14}
T=2T/(1-T^2).
\end{equation}
This only has solutions $T=0,i$ and $-i$. As $K$ and $Y$ are real this implies $T$ is real also, hence the only possible solution left is $T=0$. Substituting \eqref{fixeq8} into \eqref{fixeq14} for $T=0$ we see we must have
\begin{equation}\label{fixeq15}
\frac{K-Y^2}{Y(K+1)}=0,
\end{equation}
which implies $K=Y^2$. We can express $Y^2$ in terms of $K$ by
\begin{equation}\label{fixeq16}
Y^2=\tan^2[\cos^{-1}X]=(1-X^2)/X^2=X^{-2}-1=K(K-2)/(2K-1).
\end{equation}
Substituting \eqref{fixeq16} into \eqref{fixeq15} and rearranging we see
\[K(K+1)=0,\]
which implies $K=0$ or $K=-1$. However $K=0$ and $K=-1$ are not valid values of $K$; hence \eqref{fixeq1},\eqref{fixeq2} and \eqref{fixeq3} are never satisfied simultaneously, contradicting $\phi_1=-\phi_2$.
\end{proof}

We have to deal with the case where $H$ has three fixed rays. It is clear that it is not possible for $\cos \phi_i$ to be the same for all three fixed rays, but we find a condition under which they are all different.

\begin{lemma}\label{s4lT2}
Let $\theta\neq0$ and $K>1$. If $H$ has three fixed rays $R_{\phi_i}$ satisfying $\phi_2<\phi_0<\phi_1$, as in Lemma~\ref{s4l9}, then
\[\phi_1\neq-\phi_0.\]
Further if
\[\theta\geq\pi/6 \mbox{ then } \phi_i\neq-\phi_j \mbox{ for all } i\neq j.\]
However if $\theta<\pi/6$ then there exists some $K$ such that
\[\phi_i=-\phi_2.\]
for $i=0$ or $i=1$.
\end{lemma}

\begin{proof}
As $\phi_1,\phi_0>0$ we must have $\phi_1\neq-\phi_0$. Suppose $\phi_2=-\phi_0$. Recall from Lemma~\ref{s4l1} that for this to be possible
\[\widetilde{\mathbb{F}}_\theta^+\cap-\widetilde{\mathbb{F}}_\theta^-\neq\emptyset.\]
This implies that $\phi_1$ must satisfy the two inequalities
\[2\theta<\phi_1<\pi/2+\theta \mbox{ and } 0<\phi_1<\pi/2-\theta,\]
which implies
\[0<2\theta<\pi/2-\theta \Rightarrow 0<\theta<\pi/6.\]

If $\theta<\pi/6$ then, by Lemma~\ref{s4lT1}, when $K=K_\theta$ (recall Lemma~\ref{s4l7}) we know $\phi_i\neq\phi_2$, for $i=0,1$. For $K>K_{\theta}$, the neutral fixed point splits into two fixed points $\phi_0$ and $\phi_1$. Further $\phi_1\to\pi/2+\theta$ and $\phi_0\to2\theta$ as $K\to\infty$; also $\phi_2\to-\pi/2+\theta$ as $K\to\infty$. Hence by continuity there must exist some $K>K_\theta$ such that $\phi_2=-\phi_i$ for $i=0$ or $i=1$.
\end{proof}

The previous lemmas show that if $H_2$ has one fixed ray, then if $H_1$ has two or three fixed rays, $H_1$ and $H_2$ cannot be quasiconformally equivalent on a neighbourhood of infinity.

\subsection{The case where $H_2$ has two fixed rays}

We move on to the case where both $H_1$ and $H_2$ have more than one fixed ray. To start, we will show that if there is a quasiconformal equivalence between $H_1$ and $H_2$, it must map the immediate basin of attraction of the non-repelling fixed ray of $H_1$ into the immediate basin of attraction of the non-repelling fixed ray of $H_2$.

\begin{lemma}\label{s7basin}
If $H_1$ and $H_2$ have immediate basins of attraction $\Lambda_1^*$ and $\Lambda_2^*$ respectively for the respective non-repelling fixed rays, and are quasiconformally equivalent in a neighbourhood $U$ of infinity via the map $\Psi$, then $\Psi(\overline{\Lambda_1^*}\cap U)=\overline{\Lambda_2^*}\cap \Psi(U)$.
\end{lemma}

\begin{proof}
Since $\overline{\Lambda_1^*}$ is fixed by $H_1$, we have
\[ H_1( \overline{\Lambda_1^*} \cap U ) = \overline{\Lambda_1^*} \cap H_1(U).\]
Since $\Psi$ is injective,
\[ \Psi ( H_1 ( \overline{\Lambda_1^*} \cap U )) = \Psi ( \overline{\Lambda_1^*} ) \cap \Psi(H_1(U)).\]
Using the quasiconformal equivalence,
\[ H_2 ( \Psi ( \overline{\Lambda_1^*} \cap U )) = \Psi ( \overline{\Lambda_1^*} ) \cap H_2 ( \Psi (U)),\]
but we also have
\[ H_2 ( \Psi ( \overline{\Lambda_1^*} \cap U )) \subset H_2 ( \Psi (\overline{\Lambda_1^*})) \cap H_2 ( \Psi (U)).\]
Therefore, in a neighbourhood $U'$ of infinity, we have
\[ \Psi ( \overline{\Lambda_1^*} ) \cap U' \subset H_2 (\Psi(\overline{\Lambda_1^*}) ) \cap U'.\]
This argument also shows that in a neighbourhood $U_n'$ of infinity, we have
\begin{equation}
\label{aboringequation}
\Psi ( \overline{\Lambda_1^*} ) \cap U_n' \subset H_2^n (\Psi(\overline{\Lambda_1^*}) ) \cap U_n'
\end{equation}
for any $n \in \N$.

Now $\Psi ( \overline{\Lambda_1^*} )$ cannot spiral, as in that case, it would intersect all fixed rays of $H_2$. Then we can apply Lemma~\ref{s7ltend} to see the corresponding traces squared of the M\"{o}bius maps $A_i$ of the fixed ray $R_{\phi_i}$ and $B_j$ of the fixed rays $R_{\psi_j}$ must be equal. However Lemmas~\ref{s4lT1} and \ref{s4lT2} show that this cannot be the case. Therefore $\Psi (\overline{\Lambda_1^*} )$ must be contained in some sector.
By \eqref{aboringequation} and Theorem \ref{s2t2}, we must have that
\[ \Psi ( \overline{\Lambda_1^*}\cap U ) \subset \overline{\Lambda_2^*} \cap \Psi(U) \]
in a neighbourhood of infinity.
The same argument applied to $\Psi^{-1}$ shows that
\[ \Psi^{-1}(\overline{\Lambda_2^* } \cap \Psi(U)) \subset \overline{\Lambda_1^* } \cap U, \]
and the lemma is proved.
\end{proof}

We will next show that if $H_2$ has two fixed rays and $H_1$ has three fixed rays, then there cannot be a quasiconformal equivalence between them.

\begin{lemma}\label{s7twothree}
Let $H_1,H_2$ have three and two fixed rays respectively. Then there cannot be a quasiconformal equivalence between them in any neighbourhood of infinity.
\end{lemma}

\begin{proof}
Suppose that $H_2$ has fixed rays $R_{\psi_1},R_{\psi_2}$ with $\psi_1>\psi_2$ and $H_1$ has fixed rays $R_{\phi_i}$, $i=1,2,3$ with $\phi_2<0<\phi_0<\phi_1$.
Suppose for a contradiction that there is a quasiconformal equivalence $\Psi$ between them. Then by Lemma \ref{s7basin}, we have $\Psi(\overline{\Lambda_1^*}\cap U)=\overline{\Lambda_2^*}\cap \Psi(U)$,
where $\overline{\Lambda_1^*}$ is the closed sector bounded by the rays $R_{\phi_2}$ and $R_{\phi_1}$, and $\overline{\Lambda_2^*}$ is the closed sector bounded by the rays $R_{\psi_1}$ and $R_{\psi_2}$.

We can lift these sectors to the strip
\[\mathbb{S}:=\{z\in\C\;|\; -\pi/2\leq\Im(z)\leq\pi/2\}\]
via the quasiconformal maps $F_i: \overline{\Lambda_i^*} \to \mathbb{S}$ given by
\[F_1(re^{is}) = \left\{          \begin{array}{ll}
                                 \log r +i \left ( \frac{\pi(s - \phi_0)}{2(\phi_1-\phi_0 )} \right ) & \mbox{if } \phi_0\leq s \leq\phi_1,\\
                                 \log r - i \left ( \frac{ \pi ( s - \phi_0 ) }{2(\phi_2 - \phi_0 )} \right ) & \mbox{if } \phi_2\leq s <\phi_0.
                               \end{array}
                             \right. ,\]
and
\[ F_2(re^{is}) = \log r  + i \left ( \frac{ \pi ( 2s - (\psi_1+\psi_2) ) }{2(\psi_1 - \psi_2 )} \right ).\]
Then $\Psi : \overline{\Lambda_1^*} \cap U \to \overline{\Lambda_2^*} \cap \Psi(U)$ lifts to a quasiconformal map $P:\Omega_1 \to \Omega_2$, where $\Omega_i = F_i(\overline{\Lambda_i^*} ) \subset \mathbb{S}$ for $i=1,2$ and satisfies
$P \circ F_1 = F_2 \circ \Psi$. Note that $\Omega_i$ is a connected subset of $\mathbb{S}$ whose boundary consists of two semi-infinite lines contained in the boundary of $\mathbb{S}$, and a curve $\gamma_i$ in $\mathbb{S}$ connecting them. See figure~\ref{fStrip}.

We want to extend $P$ to a quasiconformal map from $\mathbb{S}$ to itself. There are many ways to do this, and we outline one here. Let $T_i$ be the triangle $\mathbb{S} \setminus \Omega_i$ with vertices at the endpoints of $\gamma_i$ and at $-\infty$. Define $q:\partial T_1 \to \partial T_2$ by translation on the respective horizontal semi-infinite lines, and agreeing with $P$ on $\gamma_1$.

Let $g_i:T_i \to \D$ be conformal maps of the triangles onto the disk, sending the respective vertices to $-1,i,1$ respectively. Then $\widetilde{q} = g_2 \circ q \circ g_1^{-1}$ from $S^1$ to itself is a quasisymmetric map by construction. Extend to a quasiconformal map $\widetilde{q}:\D \to \D$ via, for example, the Douady-Earle extension (see for example \cite{FM}). Then we may extend $P$ on the strip $\mathbb{S}$ by setting $P = g_2^{-1} \circ \widetilde{q} \circ g_1$ on $T_1$. This extension of $P$ is a quasiconformal map by construction.

Now, consider the attracting fixed ray $R_{\phi_0}$ of $H_1$ which is contained in the interior of the region $\Lambda_1^* \cap U$. The image of $R_{\phi_0}$ under $\Psi$ must be contained in $\overline{\Lambda_2^*}$
by Lemma \ref{s7basin}. Then by \eqref{s7e3}
\begin{equation}\label{s7etend}
\arg[H_2^n(\Psi(z))]\rightarrow \psi_1
\end{equation}
as $n \to \infty$, for $z \in R_{\phi_0}$, since all points in $\Lambda_2^*$ converge to the neutral fixed ray $R_{\psi_1}$ of $H_2$.
In particular, by lifting to the strip, $F_1(R_{\phi_0})$ is contained in the real line, but $P(F_1(R_{\phi_0})) = F_2(\Psi(R_{\phi_0}))$ is a curve which converges to the upper boundary component $\{ \Im z = \pi/2\}$ of $\mathbb{S}$.

This contradicts the lemma below applied to $P$, completing the proof.
\end{proof}

\begin{figure}[h]
\begin{center}
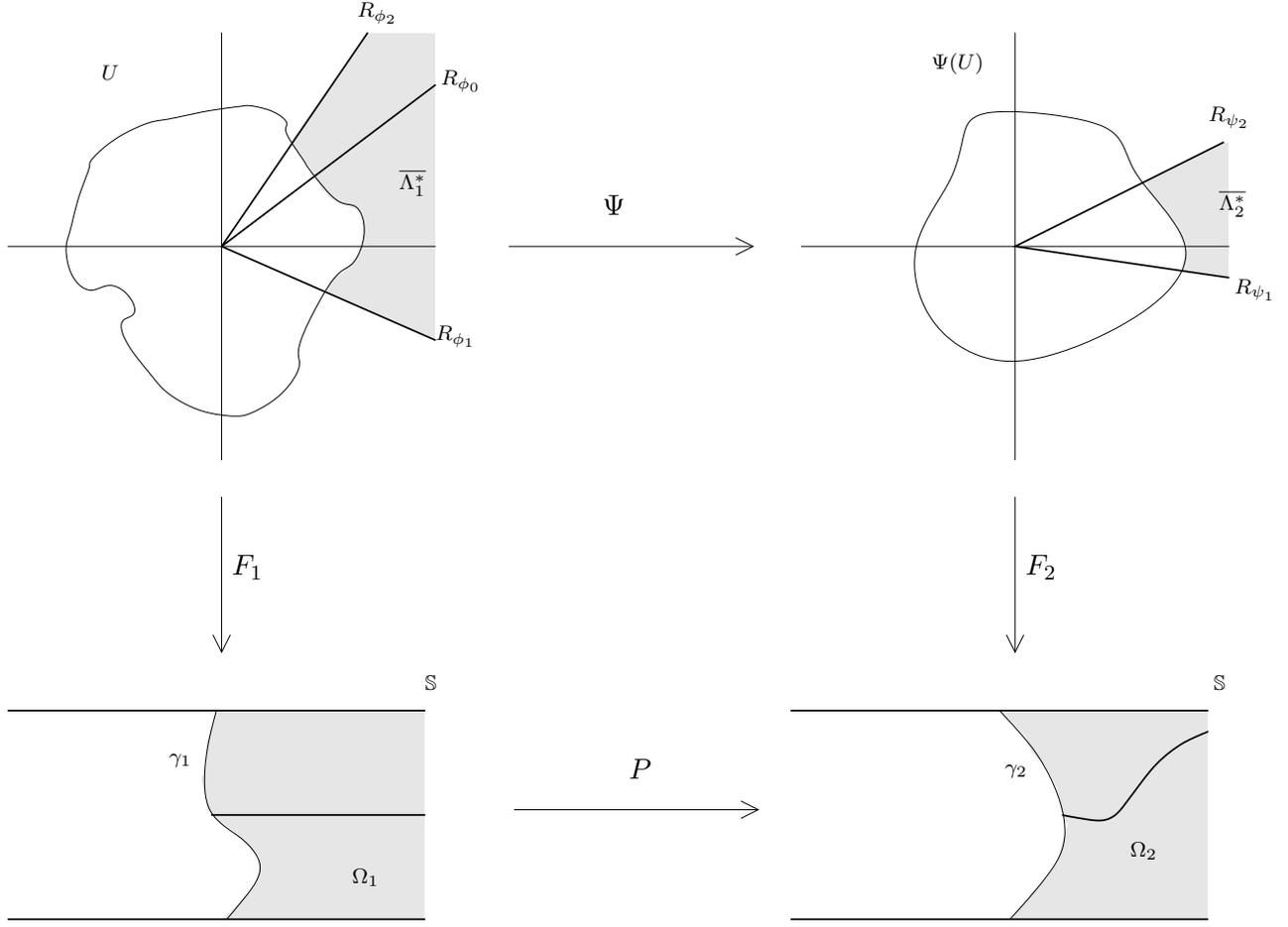
\caption{Diagram showing how $P$ is induced from the action of $\Psi$ on the sector $\overline{\Lambda_1^*}$.}\label{fStrip}
\end{center}
\end{figure}

\begin{lemma}
Let $f:\mathbb{S} \to \mathbb{S}$ be a $K$-quasiconformal map which sends $\pm \infty$ to $\pm \infty$ respectively. Then there exists $\delta < \pi/2$ such that $f(\R)$ is contained in the sub-strip $\{ z :|\Im z| < \delta \}$ of $\mathbb{S}$.
\end{lemma}

\begin{proof}
This is a strip version of a well-known result in the disk, and using the fact that $\R$ is a geodesic in $\mathbb{S}$. More specifically,
by Theorem~4.3.2 of \cite{FM}, if $f:\D \to \D$ is $K$-quasiconformal, there exists some $a$, depending on $K$, such that $f$ is a $(K,a)$-quasi-isometry. Then by Lemma~4.3.1 of \cite{FM}, given a geodesic $\gamma \subset \D$, there exists $C>0$ depending on $K$ such that $f(\gamma)$ is contained in a $C$-neighbourhood of some geodesic $\gamma'$. Lifting to the strip, $\gamma = \R$ and the corresponding $\gamma'$ is also $\R$. This proves the lemma.
\end{proof}

\subsection{Proof of Theorem \ref{s2t4}}

By Lemmas \ref{s7l3} and \ref{s7twothree}, we know that if $H_1$ and $H_2$ are quasiconformally equivalent in a neighbourhood of infinity, then they must have the same number of fixed rays. In the next two lemmas, we show that under a quasiconformal equivalence, the image of a fixed ray of $H_1$ must either intersect or approach a fixed ray of $H_2$.

\begin{lemma}\label{s7l5}
Suppose $H_1$ and $H_2$ are quasiconformally equivalent on a neighbourhood of infinity and both have one fixed ray $R_{\phi}$ and $R_{\psi}$ respectively. Then there exists $z\in R_{\phi}\cap U$ such that $\Psi(z)\in R_{\psi}$.
\end{lemma}

\begin{proof}
If $\Psi(R_{\phi})$ is a ray then the result follows from \eqref{s7e3}, using the same argument as in the proof of Lemma~\ref{s7l3}. Suppose $\Psi(R_{\phi})$ is not a ray, then it must intersect a sector $\Delta$. By Theorem~\ref{s2t2}
\[\Delta\cap\{ H_2^{-k}(R_{\psi}) \} \neq\emptyset,\]
hence there exists some ray $R\subset\Delta$ and $n\in\N$ such that $(H_2)^n(R)=R_\psi$. We can then choose $w\in R_\phi$ such that $\Psi(w)\in R$. From \eqref{s7e3} we know
\[(H_2)^n(\Psi(w))=\Psi((H_1)^n(w)).\]
Choosing $z=(H_1)^n(w)$ completes the proof.
\end{proof}

\begin{lemma}\label{s7l6}
Suppose $H_1$ and $H_2$ are quasiconformally equivalent on a neighbourhood of infinity and have three fixed rays. Let $R_{\phi_0}$ and $R_{\psi_0}$ be the attracting fixed rays of $H_1$ and $H_2$ respectively. Then for $z\in R_{\phi_0}$
\[\arg[H_2^n (\Psi(z))]\to \psi_0 \mbox{ as } n\to\infty.\]
The remaining fixed rays $R_{\phi_i}$ for $i=1,2$ of $H_1$ and $R_{\psi_i}$ for $i=1,2$ of $H_2$ such that $\phi_2<\phi_0<\phi_1$ and $\psi_2<\psi_0<\psi_1$ must satisfy $\Psi(R_{\phi_i})=R_{\psi_i}$ for $i=1,2$.
\end{lemma}

\begin{proof}
We know from Lemma~\ref{s7basin} that
\[\Psi(\overline{\Lambda_1^*}\cap U) = \overline{\Lambda_2^*}\cap U.\]
By Lemma~\ref{s4l11} we know $\widetilde{\Lambda}_1^*=(\phi_1,\phi_2)$ and $\widetilde{\Lambda}_2^*=(\psi_1,\psi_2)$, proving the final part of the lemma.

If $\Psi(R_{\phi_0})$ is a ray then $\Psi(R_{\phi_0})=R_{\psi_0}$ from \eqref{s7e3}, by using the same argument as in Lemma~\ref{s7l3} and the fact that we already know $\Psi(R_{\phi_i})=R_{\psi_i}$ for $i,j=1,2$. Assume $\Psi(R_{\phi_0})$ is not a ray, then by Theorem~\ref{s2t2}
\[\Psi(R_{\phi_0})\cap\Lambda_2\neq\emptyset.\]
Choosing $w\in\Psi(R_{\phi_0})\cap\Lambda_2$ implies $\arg[H_2^n(w)]\to \psi_0$; choosing $z=\Psi^{-1}(w)$ proves the lemma.
\end{proof}

Now finally we piece everything together.
\begin{proof}[Proof of Theorem \ref{s2t4}]
Suppose that there is a quasiconformal equivalence $\Psi$ between $H_1$ and $H_2$ on some neighbourhood $U$ of infinity.
First notice that any neighbourhood of infinity intersects every ray. In particular, it intersects fixed rays of $H_1$ and $H_2$. By Lemma \ref{s7l3} and \ref{s7twothree}, $H_1$ and $H_2$ must have the same number of fixed rays. Each fixed ray $\phi_i$ and $\psi_j$ of $H_1$ and $H_2$ respectively has a corresponding M\"{o}bius map $A_i,B_j$ respectively.

Suppose $H_1$ and $H_2$ have one fixed ray. Then Lemma~\ref{s7l5} tells us that we contradict Lemma~\ref{s7ltend} unless $\tr(A_1)=\tr(B_1)$.

Suppose $H_1$ and $H_2$ have two fixed rays $R_{\phi_i}$ and $R_{\psi_j}$ respectively, where $\phi_2<\phi_1$ and $\psi_2<\psi_1$. Lemma~\ref{s7basin} implies $\Psi(R_{\phi_i})=R_{\psi_i}$ for $i=1,2$, and so we contradict Lemma~\ref{s7ltend} unless $\tr(A_i)=\tr(B_i)$ for both $i=1,2$.

Finally suppose $H_1$ and $H_2$ have three fixed rays. Again, by Lemma~\ref{s7l6} we contradict Lemma~\ref{s7ltend} unless $\tr(A_i)=\tr(B_i)$ for $i=1,2,3$.
\end{proof}

We can do better than this and rule out further possibilities, as stated in Corollary~\ref{s2ct4}.

\begin{proof}[Proof of Corollary \ref{s2ct4}]

Suppose $\theta\in(0,\pi/2)$ is fixed. Recall from Lemma~\ref{s4l5} that there always exists a fixed point $\phi\in\widetilde{\mathbb{F}}_\theta^-$ of $\widetilde{H}$, that is $0>\phi>\theta-\pi/2$. Let $K_1,K_2>1$, $\widetilde{h}_1:=\widetilde{h}_{K_1,\theta}$, $\widetilde{h}_2:=\widetilde{h}_{K_2,\theta}$ and $0>\varphi>\theta-\pi/2$. Then
\begin{equation}\label{C1}
\widetilde{h}_1(\varphi)-\widetilde{h}_2(\varphi)=\tan^{-1}\left(\frac{\tan(\varphi-\theta)}{K_1}\right)+\theta- \tan^{-1}\left(\frac{\tan(\varphi-\theta)}{K_2}\right)-\theta.
\end{equation}
Using the addition formula for $\tan^{-1}$ and simplifying, \eqref{C1} becomes
\begin{equation}\label{C2}
\widetilde{h}_1(\varphi)-\widetilde{h}_2(\varphi) = \tan^{-1}\left(\frac{(K_2-K_1)\tan(\varphi-\theta)}{K_1K_2+\tan^2(\varphi-\theta)}\right).
\end{equation}
Here $\tan(\varphi-\theta)<0$ is fixed. Hence if $K_1>K_2$ then \eqref{C2} implies $\widetilde{h}_1(\varphi)>\widetilde{h}_2(\varphi)$. Let $\phi_K$ be the fixed point of $\widetilde{h}_{K,\theta}$ in $\widetilde{\mathbb{F}}_\theta^-$. By Lemma~\ref{s4l8} we know $\widetilde{h}_1(\phi_{K_1})=\phi_{K_1}/2$ and $\widetilde{h}_1(\varphi)<\widetilde{h}_1(\varphi)/2$ for $0<\varphi<\phi_{K_1}$. By \eqref{C2} if $K_2>K_1$ then $\widetilde{h}_2(\phi_{K_1})<\phi_{K_1}/2$ and $\widetilde{h}_2(\varphi)<\widetilde{h}_2(\varphi)/2$ for $0<\varphi<\phi_{K_1}$. Hence it must be the case that $\phi_{K_2}<\phi_{K_1}$. This shows that $\phi_K$ decreases as $K$ increases.

Suppose that $H_{K_1,\theta}$ is quasiconformally conjugate to $H_{K_2,\theta}$. Then by Theorem~\ref{s2t4} they must have the same number of fixed rays and the corresponding traces squared, for the fixed points $\phi_{K_1},\phi_{K_2}\in\widetilde{\mathbb{F}}_\theta^-$ of $\widetilde{H}_{K_1,\theta}$ and $\widetilde{H}_{K_2,\theta}$ respectively, must be equal. This implies
\[\frac{(K_1+1)^2}{2K_1}(1+\cos\phi_{K_1})=\frac{(K_2+1)^2}{2K_2}(1+\cos\phi_{K_2}).\]
Rearranging we obtain
\begin{equation}\label{ct4e1}
\frac{K_2(K_1+1)^2}{K_1(K_2+1)^2}=\frac{1+\cos\phi_{K_1}}{1+\cos\phi_{K_2}}.
\end{equation}
Suppose $K_1<K_2$, then we know $\phi_{K_1}>\phi_{K_2}$ which implies the right hand side of \eqref{ct4e1} is greater than 1, however
\[K_2(K_1+1)^2-K_1(K_2+1)^2=(K_2-K_1)(1-K_1K_2)<0.\]
Hence the left hand side of \eqref{ct4e1} is less than 1, a contradiction.

If $\theta=0$ then $\phi_0=0$ is always a fixed point of $\widetilde{H}_{K_1,0}$ and $\widetilde{H}_{K_2,0}$ for any $K_1,K_2>1$. Hence the right hand side of \eqref{ct4e1} is always equal to one, but the left hand side is only equal to one if $K_1=K_2$.
\end{proof}

\section{Properties of hyperbolic M\"{o}bius maps}

\begin{lM1}
Let $A:\D\to\D$ be a hyperbolic M\"{o}bius map of the disk. Then $A$ is conjugate to a M\"{o}bius map of the upper half plane of the form
\[\widehat{A}(z)=kz,\]
where
\begin{equation}\label{s8Tr}
k=(T-2-(T^2-4T)^{\frac{1}{2}})/2,
\end{equation}
where $T:=\tr^2(A).$
\end{lM1}

\begin{proof}
We know that $\tr A^2>4$ and so by standard hyperbolic geometry we can lift to $\widetilde{A}:\Hp\to\Hp$. We see $\widetilde{A}$ has the same trace as $A$ and is of hyperbolic type so must be conjugate to $\widehat{A}(z)=kz$ for some $k>0$. Conjugation preserves trace hence
\[k+1/k+2=\tr^2(\widehat{A})=\tr^2(A)=T,\]
solving this for $k$ and taking the negative square root gives equation \eqref{s8Tr} and $k<1$. Taking the positive square root would give the reciprocal.
\end{proof}

\begin{tM3}
Let $A,A_j:\D\to\D$ be hyperbolic M\"{o}bius maps such that $A^n(z)\to\alpha\in\partial\D$ as $n\to\infty$ and $A_j\to A$ locally uniformly as $j\to\infty$. Let
\[t_n(z)=A_1\circ A_2\circ\ldots\circ A_n(z).\]
Then
\[d_h(0,t_n(z))=\log\left[\frac{1}{\prod_{j=1}^n k_j}\right] + O(1),\]
for large $n$, where $d_h$ denotes the hyperbolic metric on $\D$, $k_j<1$ for all $j$ and $k_j\to k$, where $k_j,k$ are the quantities defined in Lemma~\ref{sMl1}.
\end{tM3}

\begin{proof}
First if $A^n(z)\to\alpha$ for $z\in\D$ then $t_n(z)\to \alpha$ by Lemma~\ref{sMlMM}. Now let $B=A^{-1}$ and $B_j=A_j^{-1}$. Then if $\alpha,\beta\in\partial\D$ are the attracting and repelling fixed points of $A$ respectively, then $\beta$ is the attracting fixed point and $\alpha$ is the repelling fixed point of $B$.
Similarly if $\alpha_j,\beta_j\in\partial\D$ are the attracting and repelling fixed points of $A_j$ respectively, then $\beta_j$ is the attracting fixed point and $\alpha_j$ is the repelling fixed point of $B_j$. Further, we have $B_j\to B$ and so $\alpha_j\to\alpha$ and $\beta_j\to\beta$ as $j\to\infty$.

We write $\widetilde{B}$ for the lift of $B$ to $\Hp$ via $\gamma:\D\to\Hp$
so that $\widetilde{B} = \gamma \circ B \circ \gamma^{-1}$.
We choose $\gamma$ so that $\gamma(\alpha)=\infty$ and $\gamma(0)=i$. This then means that $\gamma(\beta)=X\in\R$ and $\gamma(\beta_j)=X_j\in\R$, where $X_j\to X$ as $j\to\infty$.

We can also conjugate by the maps $\Phi,\Phi_i:\Hp\to\Hp$ where $\Phi(z)=z-X$ and $\Phi_i(z)=z-X_i$. This means that
\[\widetilde{B}(z)=\Phi^{-1}\circ\widehat{B}\circ\Phi(z) \mbox{ and } \widetilde{B}_j(z)=\Phi_j^{-1}\circ\widehat{B}_j\circ\Phi_j(z),\]
where $\widehat{B}(z)=kz$ and $\widehat{B}_j(z)=k_jz$. The factors $k$ and $k_j$ are determined as in Lemma~\ref{sMl1} so that $k,k_j<1$ and $k_j\to k$ as $j\to\infty$. Let
\[s_n:=t_n^{-1}=B_n\circ\ldots\circ B_1.\]
Writing $\rho_{\Hp}$ for the hyperbolic metric on $\Hp$, by conformal invariance we have
\begin{align*}
d_h(0,t_n(z))&=\rho_{\Hp}(i,\gamma(t_n(z)) )\\
&=\rho_{\Hp}(i,\widetilde{t_n}(\gamma(z ) )\\
&=\rho_{\Hp}(\widetilde{s_n}(i), \gamma(z) ).
\end{align*}
We can rewrite $\widetilde{s}_n(i)$ as
\[\widetilde{s}_n(i)=\Phi_n^{-1}\circ\widehat{B}_n\circ\Phi_n\circ\Phi_{n-1}^{-1}\circ\widehat{B}_{n-1}\circ\Phi_{n-1}\circ\ldots\circ\Phi_1^{-1}\circ\widehat{B}_1\circ\Phi_1(i).\]
It is not hard to see that
\[\widetilde{s}_n(i)=\sum_{j=1}^n\left((X_{j-1}-X_j) \left (\prod_{m=j}^n k_m \right ) \right)+X_n + i\prod_{m=1}^n k_m,\]
where we use the convention $X_0=0$. Writing
\[P_n=\prod_{m=1}^n k_m, \quad R_n=\sum_{j=1}^n\left((X_{j-1}-X_j) \left ( \prod_{m=j}^n k_m \right )\right)\]
and $\gamma(z)=x+iy\in \Hp$, we see that since $X_n \to X$ and $\widetilde{s_n}(i) \to X$ by Lemma~\ref{sMlMM}, we have $R_n \to 0$ as $n \to \infty$.
By the formula for the hyperbolic metric in $\Hp$,
\begin{align}
\rho_{\Hp}(\widetilde{s}_n(i),\gamma(z))&= \cosh^{-1}\left(\frac{(R_n+X_n-x)^2+(P_n-y)^2}{2P_n y}\right)\notag\\
&= \cosh^{-1}\left(\frac{P_n}{2y} + \frac{R_n^2+X_n^2+x^2-2xR_n-2xX_n+2X_nR_n-y}{2P_n}\right).\label{Ds_n1}
\end{align}
Since $R_n \to 0$, $X_n \to X$ and $x,y$ are fixed,
\begin{equation}\label{Ds_n2}
(R_n^2+X_n^2+x^2-2xR_n-2xX_n+2X_nR_n-y)\longrightarrow (X^2+x^2-2xX-y),
\end{equation}
as $n \to \infty$.
This expression is bounded. We also have
\begin{equation}\label{Ds_n3}
\frac{P_n}{2y}\to 0 \mbox{ as } n\to\infty.
\end{equation}
Hence, from \eqref{Ds_n1},\eqref{Ds_n2}, \eqref{Ds_n3} and using the identity $\cosh^{-1}(z)=\log(z+\sqrt{z^2+1})$, we can write
\begin{align}
\rho_{\Hp}(\widetilde{s}_n(i),\gamma(z))&= \cosh^{-1}\left[ O\left ( \frac{1}{P_n} \right ) \right]\notag\\
&=\log\left[ O \left ( \frac{1}{P_n} \right ) \right]\\
&=\log \left ( \frac{1}{P_n} \right ) +O(1),
\end{align}
which proves the lemma.
\end{proof}

\end{document}